\documentclass[11pt]{article}

\usepackage{amsmath,amsthm,amssymb}
\usepackage{amsmath,amssymb,enumitem}
\usepackage{algorithm}
\usepackage{algorithmic}[1]
\usepackage{float}
\usepackage{subfig}
\usepackage{titlesec}
\usepackage{scalerel}
\usepackage{epstopdf}
\usepackage{graphicx}
\usepackage{xspace}
\usepackage{comment}
\usepackage{url}
\usepackage[a4paper]{geometry}
\def \endproof{\enskip \null \nobreak \hfill \openbox \par}

\newcommand{\RR}{\mathbb{R}}
\newcommand{\CC}{\mathbb{C}}

\newcommand{\dist}{\mbox{dist}}

\newcommand{\eg}{{\it e.g.}}
\newcommand{\ie}{{\it i.e.}}

\DeclareMathOperator*{\argmin}{argmin}

\DeclareMathOperator*{\Argmin}{Argmin}

\newtheorem{lem}{Lemma}
\newtheorem{thm}{Theorem}
\newtheorem{coro}{Corollary}
\newtheorem{defi}{Definition}
\newtheorem{prop}{Proposition}

\newtheorem{ass}{Assumption}
\newtheorem{fact}{Fact}

\newtheorem{algo}{Algorithm}

\begin{document}
\title{\textbf{On the Quadratic Convergence of the Cubic Regularization Method under a Local \\ Error Bound Condition}}
\author{
Man-Chung Yue
\thanks{Imperial College Business School, Imperial College London, United Kingdom. E-mail: {\tt
m.yue@imperial.ac.uk}}
\and
Zirui Zhou%
\thanks{Department of Mathematics, Simon Fraser University, Canada. E-mail: {\tt ziruiz@sfu.ca} . This author is supported by an NSERC Discovery Grant and the SFU Alan Mekler postdoctoral fellowship.}
\and
Anthony Man-Cho So
\thanks{Department of Systems Engineering and Engineering Management, The Chinese University of
Hong Kong, Shatin, N. T., Hong Kong. E-mail: {\tt manchoso@se.cuhk.edu.hk}}
}
\date{}
\maketitle
\thispagestyle{empty}

\begin{abstract}
\medskip

In this paper we consider the cubic regularization (CR) method for minimizing a twice continuously differentiable function. While the CR method is widely recognized as a globally convergent variant of Newton's method with superior iteration complexity, existing results on its local quadratic convergence require a stringent non-degeneracy condition. We prove that under a local error bound (EB) condition, which is much weaker a requirement than the existing non-degeneracy condition, the sequence of iterates generated by the CR method converges at least Q-quadratically to a second-order critical point. This indicates that adding a cubic regularization not only equips Newton's method with remarkable global convergence properties but also enables it to converge quadratically even in the presence of degenerate solutions. As a byproduct, we show that without assuming convexity, the proposed EB condition is equivalent to a quadratic growth condition, which could be of independent interest. To demonstrate the usefulness and relevance of our convergence analysis, we focus on two concrete nonconvex optimization problems that arise in phase retrieval and low-rank matrix recovery, respectively, and prove that with overwhelming probability, the sequence of iterates generated by the CR method for solving these two problems converges at least Q-quadratically to a global minimizer. We also present numerical results of the CR method when applied to solve these two problems to support and complement our theoretical development.

\bigskip


\bigskip

\end{abstract}

%\begin{keywords}
%	
%\end{keywords}

\section{Introduction}\label{sec:introduction}
Consider the unconstrained minimization problem
\begin{equation}\label{opt:smooth_nc}
\min_{x\in\RR^n} f(x),
\end{equation}
where $f:\RR^n\rightarrow\RR$ is assumed to be twice continuously differentiable.
Newton's method is widely regarded as an efficient local method for solving problem \eqref{opt:smooth_nc}. The cubic regularization (CR) method, which is short for cubic regularized Newton's method, is a globally convergent variant of Newton's method. Roughly speaking, given the current iterate $x^k$, the CR method determines the next one by minimizing a cubic regularized quadratic model of $f$ at $x^k$; \ie,
\begin{equation}
\label{eq:update-CR}
x^{k+1} \in \Argmin_{x\in\RR^n} \left\{ f(x^k) + \nabla f(x^k)^T(x - x^k) + \frac{1}{2}(x - x^k)^T\nabla^2f(x^k)(x - x^k) + \frac{\sigma_k}{6}\|x - x^k\|^3 \right\},
\end{equation}
where the regularization parameter $\sigma_k>0$ is chosen such that $f(x^{k+1}) \leq f(x^k)$. The idea of using cubic regularization first appeared in Griewank \cite{G81}, where he proved that any accumulation point of $\{x^k\}_{k\ge0}$ generated by the CR method is a second-order critical point of $f$; \ie, an $x\in\RR^n$ satisfying $\nabla f(x) = 0$ and $\nabla^2 f(x) \succeq 0$. Later, Nesterov and Polyak~\cite{NP06} presented the remarkable result that the CR method has a better global iteration complexity bound than that for the steepest descent method. Elaborating on these results, Cartis {\it et al.}~\cite{CGT11a,CGT11b} proposed an adaptive CR method for solving problem~\eqref{opt:smooth_nc}, where $\{\sigma_k\}_{k\ge0}$ are determined dynamically and subproblems \eqref{eq:update-CR} are solved inexactly. They showed that the proposed method can still preserve the good global complexity bound established in \cite{NP06}. Based on these pioneering works, the CR method has been attracting increasing attention over the past decade; see, \eg,~\cite{CGT10,T13,Y15} and references therein.

In addition to these global convergence properties, the CR method, as a modified Newton's method, is also expected to attain a fast local convergence rate. It is known that if any accumulation point $\bar{x}$ of the sequence $\{x^k\}_{k\ge0}$ generated by the CR method satisfies 
\begin{equation}
\label{eq:sec-suff-cond}
\nabla f(\bar{x})=0 \quad \mbox{and} \quad \nabla^2 f(\bar{x}) \succ 0,
\end{equation} 
then the whole sequence $\{x^k\}_{k\ge0}$ converges at least Q-quadratically to $\bar{x}$; see~\cite[Theorem 4.1]{G81} or~\cite[Theorem 3]{NP06}.\footnote[2]{A sequence of vectors $\{w^k\}_{k\ge0}$ in $\RR^n$ is said to converge {\it Q-quadratically} to a vector $w^\infty$ if there exists a positive constant $M$ such that $\|w^k - w^\infty\|/\|w^k - w^\infty\|^2 \leq M$ for all sufficiently large $k$; see, \eg, \cite[Appendix A.2]{NW06}.} Nevertheless, the non-degeneracy condition \eqref{eq:sec-suff-cond} implies that $\bar{x}$ is an isolated local minimizer of $f$ and hence does not hold for many nonconvex functions in real-world applications. For example, consider the problem of recovering a positive semidefinite matrix $X^*\in\RR^{n\times n}$ with rank $r\ll n$, given a linear operator $\mathcal{A}:\RR^{n\times n}\rightarrow\RR^m$ and a measurement vector $b = \mathcal{A}(X^*)$. A practically efficient approach for recovering $X^*$ is to solve the following nonconvex minimization problem (see, \eg, \cite{BNS16}):
$$%\begin{equation} \label{eq:ncvx-lrm}
\min_{U\in\RR^{n\times r}} f(U) := \frac{1}{4m}\|\mathcal{A}(UU^T) - b\|^2.
$$%end{equation}
Noticing that $f(U) = f(UR)$ for any $U\in\RR^{n\times r}$ and any orthogonal matrix $R\in\RR^{r\times r}$, it is not hard to see that there is no isolated local minimizer of $f$ when $r\geq 2$, which implies that there is no $U\in\RR^{n\times r}$ such that $\nabla f(U) = 0$ and $\nabla^2 f(U)\succ 0$ when $r\geq 2$. Similar degeneracy features can also be found in various nonconvex optimization formulations used in phase retrieval \cite{SQW16} and deep learning~\cite{ZL17}. In view of this, it is natural to study the local convergence properties of the CR method for solving problems with non-isolated minimizers.
Moreover, the non-degeneracy condition~\eqref{eq:sec-suff-cond} seems too stringent for the purpose of ensuring quadratic convergence of the CR method. Indeed, one can observe from \eqref{eq:update-CR} that due to the cubic regularization, the CR method is well defined even when the Hessian at hand has non-positive eigenvalues. In addition, the CR method belongs to the class of regularized Newton-type methods, many of which have been shown to attain a superlinear or quadratic convergence rate even in the presence of non-isolated solutions. For instance, Li {\it et al.}~\cite{LFQY04} considered a regularized Newton's method for solving the convex case of problem \eqref{opt:smooth_nc}. 
%Each trial step is computed by
%\begin{equation}
%\label{eq:update-rN}
%d^k = -(\nabla^2 f(x^k) + \mu_k I)^{-1}\nabla f(x^k) 
%\end{equation}
%for some regularization parameter $\mu_k>0$. 
They proved that if $f$ satisfies a local error bound condition, which is a weaker requirement than \eqref{eq:sec-suff-cond}, then the whole sequence $\{x^k\}_{k\ge0}$ converges superlinearly or quadratically to an optimal solution. Yue {\it et al.}~\cite{YZS16} extended such result to a regularized proximal Newton's method for solving a class of nonsmooth convex minimization problems. Other regularized Newton-type methods that have been shown to attain superlinear or quadratic convergence for problems with non-isolated solutions include, among others, the classic Levenberg-Marquardt (LM) method \cite{YF01,FY05} for nonlinear equations, Newton-type methods for complementarity problems \cite{T00}, and regularized Gauss-Newton methods for nonlinear least-squares~\cite{BM14}.

In this paper we establish the quadratic convergence of the CR method under the assumption of the following local error bound condition.

\begin{defi}[EB Condition] 
We say that $f$ satisfies the local error bound (EB) condition if there exist scalars $\kappa,\rho>0$ such that 
\begin{equation}
\label{eq:eb-intro}
\dist(x,\mathcal{X}) \leq \kappa\|\nabla f(x)\| \quad\mbox{whenever} \ \ \dist(x,\mathcal{X})\leq \rho,
\end{equation}
where $\mathcal{X}$ is the set of second-order critical points of $f$ and $\dist(x,\mathcal{X})$ denotes the distance of $x$ to $\mathcal{X}$.
\end{defi}

\noindent
As we shall see in Section 3, the above local EB condition is a weaker requirement than the non-degeneracy condition \eqref{eq:sec-suff-cond}. We prove that if $f$ satisfies the above local EB condition, then the whole sequence $\{x^k\}_{k\ge0}$ generated by the CR method converges at least  Q-quadratically to a second-order critical point of $f$. This, together with the pioneering works \cite{G81,NP06,CGT11a}, indicates that adding a cubic regularization not only equips Newton's method with superior global convergence properties but also enables it to converge quadratically even in the presence of degenerate solutions. We remark that our proof of quadratic convergence is not a direct extension of those from the aforementioned works on regularized Newton-type methods. In particular, a major difficulty in our proof is that the descent direction $d^k = x^{k+1} - x^k$ of the CR method is obtained by minimizing a nonconvex function, as one can see from \eqref{eq:update-CR}. By contrast, the descent directions of the regularized Newton-type methods in \cite{LFQY04,YZS16,YF01,FY05,BM14} are all obtained by minimizing a strongly convex function. For instance, the LM method for solving the nonlinear equation $F(x)=0$ computes its descent direction by solving the strongly convex optimization problem 
\begin{equation}
\label{eq:update-LM}
d^k = \argmin_{d\in\RR^n} \left\{\|F(x^k) + F^\prime(x^k)d\|^2 + \mu_k\|d\|^2 \right\},
\end{equation}
where $F^\prime$ is the Jacobian of $F$ and $\mu_k>0$ is the regularization parameter; see \cite{L44,M63}. Consequently, we cannot utilize the nice properties of strongly convex functions in our proof. Instead, we shall exploit the fact that any accumulation point of the sequence $\{x^k\}_{k\ge0}$ generated by the CR method is a second-order critical point of $f$ in our analysis. It is also worth noting that our convergence analysis unifies and sharpens those in \cite{NP06} for the so-called globally non-degenerate star-convex functions and gradient-dominated functions (see Section \ref{sec:CR} for the definitions). In particular, we show that when applied to these two classes of functions, the CR method converges quadratically, which improves upon the sub-quadratic convergence rates established in \cite{NP06}.

Besides our convergence analysis of the CR method, the proposed local EB condition could also be of independent interest. A notable feature of the EB condition \eqref{eq:eb-intro} is that its target set $\mathcal{X}$ is the set of second-order critical points of $f$. This contrasts with other EB conditions in the literature, where $\mathcal{X}$ is typically the set of first-order critical points (see, \eg,~\cite{LT93}) or the set of optimal solutions (see, \eg,~\cite{F02,ZS15}). Such feature makes our EB condition especially useful for analyzing local convergence of iterative algorithms that are guaranteed to cluster at second-order critical points. Moreover, we prove that under some mild assumptions, our local EB condition is equivalent to a quadratic growth condition (see Theorem \ref{thm:eb=qg} (ii) for the definition). Prior to this work, the equivalence between these two regularity conditions was established when $f$ is convex~\cite{AG08} or when $f$ is nonconvex but satisfies certain quadratic decrease condition \cite{DMN14}. Our result indicates that if the target set $\mathcal{X}$ is the set of second-order critical points, then the equivalence of the two regularity conditions can be established without the need of the aforementioned quadratic decrease condition.

To demonstrate the usefulness and relevance of our convergence analysis, we apply it to study the local convergence behavior of the CR method when applied to minimize two concrete nonconvex functions that arise in phase retrieval and low-rank matrix recovery, respectively. A common feature of these nonconvex functions is that they do not have isolated local minimizers. Motivated by recent advances in probabilistic analysis of the global geometry of these nonconvex functions \cite{SQW16,BNS16}, we show that with overwhelming probability, (i) the set of second-order critical points equals the set of global minimizers and (ii) the local EB condition \eqref{eq:eb-intro} holds. As a result, our analysis implies that with overwhelming probability, the sequence of iterates generated by the CR method for solving these nonconvex problems converges at least Q-quadratically to a global minimizer. Numerical results of the CR method for solving these two nonconvex problems are also presented, which corroborate our theoretical findings.

The rest of this paper is organized as follows. In Section \ref{sec:CR}, we review existing results on the convergence properties of the CR method. In Section \ref{sec:eb}, we study the local EB condition \eqref{eq:eb-intro} and prove its equivalence to a quadratic growth condition. In Section \ref{sec:LC}, we prove the quadratic convergence of the CR method under the local EB condition. In Section \ref{sec:app}, we study the CR method for solving two concrete nonconvex minimization problems that arise in phase retrieval and low-rank matrix recovery, respectively. In Section \ref{sec:numerical}, we present numerical results of the CR method for solving these two nonconvex problems. Finally, we close with some concluding remarks in Section \ref{sec:conclusion}.

\subsection{Notations}
We adopt the following notations throughout the paper.  Let $\mathbb{R}^n$ be the $n$-dimensional Euclidean space and $\langle \cdot, \cdot \rangle$ be its standard inner product. For any vector $x\in\mathbb{R}^n$, we denote by $\|x\|=\sqrt{\langle x,x\rangle}$ its Euclidean norm. Given any $\bar{x}\in\mathbb{R}^n$ and $\rho>0$, we denote by $\mathbb{B}(\bar{x};\rho)$ the Euclidean ball with center $\bar{x}$ and radius $\rho$; \ie, $\mathbb{B}(\bar{x};\rho) := \{ x\in\mathbb{R}^n: \|x - \bar{x}\| \leq \rho\}.$ For any matrix $X\in\mathbb{R}^{m\times n}$, we denote by $\|X\|$ and $\|X\|_F$ its operator norm and Frobenius norm, respectively. If in addition $X$ is symmetric, we write $\lambda_1(X)\geq\cdots\geq\lambda_n(X)$ as the eigenvalues of $X$ in decreasing order. Moreover, we write $X\succeq 0$ if $X$ is positive semidefinite. We denote by $\mathcal{O}^r$ the set of $r\times r$ orthogonal matrices; \ie, $Q^TQ = QQ^T = I_{r}$ for any $Q\in\mathcal{O}^r$, where $I_r$ is the $r\times r$ identity matrix. For any complex vector $z\in\CC^n$, we denote by $\Re(z)$ and $\Im(z)$ its real and imaginary parts, respectively. Moreover, we let $\overline{z}$ be the conjugate of $z$, $z^H = \overline{z}^T$ be the Hermitian transpose of $z$, and $\|z\| = \sqrt{z^Hz}$ be the norm of $z$. For any closed subset $C\subset\RR^n$, we denote by $\dist(x,C)$ the distance of $x\in\RR^n$ to $C$. In addition, we use $\mathcal{N}(C;\rho)$ with some $\rho>0$ to denote the neighborhood $\mathcal{N}(C;\rho):=\{x\in\RR^n: \dist(x,C)\leq \rho\}$ of $C$.

%Recall that $f:\mathbb{R}^n\rightarrow\mathbb{R}$ is the objective function in \eqref{opt:smooth_nc}, which is twice continuously differentiable. 
For any $x\in\mathbb{R}^n$, we define $ \mathcal{L}(f(x)) := \left\{ y \in \mathbb{R}^n: f(y) \leq f(x) \right\}$.
We say that $x\in\mathbb{R}^n$ is a \emph{second-order critical point} of $f$ if it satisfies the second-order necessary condition for $f$; \ie, $\nabla f(x) = 0$ and $\nabla^2 f(x) \succeq 0$. Unless otherwise stated, we use $\mathcal{X}$ to denote the set of second-order critical points of $f$ and $\mathcal{X}^*$ to denote the set of global minimizers of $f$. It is clear that $\mathcal{X}^*\subset\mathcal{X}$. Moreover, since $f$ is twice continuously differentiable, both $\mathcal{X}$ and $\mathcal{X}^*$ are closed subsets of $\mathbb{R}^n$. We assume throughout the paper that $\mathcal{X}^*$ is non-empty.

\section{The Cubic Regularization Method}\label{sec:CR}
In this section, we review the cubic regularization (CR) method for solving problem~\eqref{opt:smooth_nc} and some existing results on its convergence properties. 

Given a vector $x\in\RR^n$, we define the cubic regularized quadratic approximation of $f$ at $x$ as
\begin{equation}
\label{eq:true_cubic_model}
f_{\sigma}(p; x) = f(x) + \nabla f(x)^T(p-x) + \frac{1}{2}(p-x)^T\nabla^2 f(x)(p-x) +\frac{\sigma}{6}\|p-x\|^3,
\end{equation}
where $\sigma>0$ is the regularization parameter. In addition, we define
\begin{equation}
\label{eq:def-p-sigma}
\bar{f}_{\sigma}(x) := \min_{p\in\mathbb{R}^n} f_{\sigma}(p;x) \quad \mbox{and} \quad p_{\sigma}(x) \in \Argmin_{p\in\mathbb{R}^n} f_{\sigma}(p; x).
\end{equation}
In principle, starting with an initial point $x^0\in\mathbb{R}^n$, the CR method generates a sequence of iterates $\{x^k\}_{k\ge0}$ by letting $ x^{k+1} = p_{\sigma_k}(x^k)$ for some $\sigma_k>0$ such that
\begin{equation}
\label{eq:suff-decrease}
f(p_{\sigma_k}(x^k)) \leq \bar{f}_{\sigma_k}(x^k).
\end{equation}
Notice that this requires the computation of $p_\sigma(x)$, which is a global minimizer of $f_{\sigma}(\cdot;x)$. Although $f_{\sigma}(\cdot;x)$ is in general nonconvex, it has been shown in \cite{NP06} that $p_\sigma(x)$ can be computed by solving a one-dimensional convex optimization problem. Moreover, the optimality condition for the global minimizers of $f_\sigma(\cdot;x)$ is very similar to that of a standard trust-region subproblem~\cite[Theorem 3.1]{CGT11a}. Such observation has led to the development of various efficient algorithms for finding $p_{\sigma}(x)$ in \cite{CGT11a}. More recently, it is shown in~\cite{CD16} that the gradient descent method can also be applied to find $p_\sigma(x)$.

For the global convergence of the CR method, we need the following assumption.
\begin{ass}
	\label{ass:hess-lip}
	The Hessian of the function $f$ is Lipschitz continuous on a closed convex set $\mathcal{F}$ with $\mathcal{L}(f(x^0))\subset{\rm int}(\mathcal{F})$; \ie, there exists a constant $L>0$ such that
	\begin{equation}
	\label{eq:hess-lip}
	\| \nabla^2 f(x) - \nabla^2 f(y)\| \leq L\|x - y\|, \quad \forall x,y\in\mathcal{F}.
	\end{equation}
\end{ass}
A direct consequence of Assumption~\ref{ass:hess-lip} is that for any $x\in\mathcal{F}$, it holds that $f(p_\sigma(x)) \leq \bar{f}_{\sigma}(x)$ whenever $\sigma\geq L$ (see~\cite[Lemma 4]{NP06}). This further implies that for all $k\geq 0$, we can find a $\sigma_k \leq 2L$ such that~\eqref{eq:suff-decrease} holds. Indeed, if the Lipschitz constant $L$ is known, we can let $\sigma_k = L$. If not, by using a line search strategy that doubles $\sigma_k$ after each trial~\cite[Section 5.2]{NP06}, we can find a $\sigma_k\leq 2L$ such that~\eqref{eq:suff-decrease} holds. We now state the details of the CR method as follows.
\begin{algo}[The Cubic Regularization Method]
	\label{alg:CR}
	\normalfont
	\mbox{}
	\begin{itemize}
		\item[0.] Input an initial point $x^0\in\mathbb{R}^n$, a scalar $\bar{\sigma}\in(0,L]$, and set $k=0$.
		\item[1.] Find $\sigma_k \in[\bar{\sigma}, 2L]$ such that 
		\begin{equation}
		\label{eq:suff-decrease-copy}
		f(p_{\sigma_k}(x^k)) \leq \bar{f}_{\sigma_k}(x^k).
		\end{equation}
		\item[2.] Set $x^{k+1} = p_{\sigma_k}(x^k)$ and $k = k+1$, and go to Step 1.
	\end{itemize}
	{\bf End.}
\end{algo}

%As a consequence of Assumption~\ref{ass:hess-lip}, we have the following result~\cite[Lemma 1]{NP06}.
%\begin{lem}
%\label{lem:lip-result}
%Suppose that Assumption~\ref{ass:hess-lip} holds. Then, for any $x,y\in\mathcal{F}$, one has
%\begin{equation}
%\label{eq:lip-result-1}
%\|\nabla f(y) - \nabla f(x) - \nabla^2 f(x)(y - x)\| \leq \frac{L}{2}\|y - x\|^2
%\end{equation}
%and 
%\begin{equation}
%\label{eq:lip-result-2}
%\left|f(y) - f(x) - \langle \nabla f(x), y - x \rangle - \frac{1}{2}\langle \nabla^2 f(x) (y - x), y - x \rangle \right| \leq \frac{L}{6}\|y - x\|^3.
%\end{equation}
%\end{lem}

The following result, which can be found in \cite[Theorem 4.1]{G81} and \cite[Theorem 2]{NP06}, shows that any accumulation point of the sequence $\{x^k\}_{k\ge0}$ generated by the CR method is a second-order critical point of $f$.
\begin{fact}\label{fact:CR_GC}
Suppose that Assumption~\ref{ass:hess-lip} holds. Let $\{x^k\}_{k\ge0}$ be the sequence of iterates generated by the CR method. If $\mathcal{L}(f(x^k))$ is bounded for some $k\geq 0$, then the following statements hold.
\begin{enumerate}
	\item[(i)] $v:= \lim_{k\rightarrow\infty} f(x^k)$ exists.
	\item[(ii)] $\lim_{k\rightarrow\infty} \|x^k - x^{k-1}\| = 0$.
	\item[(iii)] The sequence $\{x^k\}_{k\ge0}$ has at least one accumulation point. Moreover, every accumulation point $\bar{x}$ of $\{x^k\}_{k\ge0}$ satisfies
	$$%\begin{equation} \label{eq:CR_GC}
	f(\bar{x})=v,\quad\nabla f(\bar{x})=0,\quad\nabla^2 f(\bar{x})\succeq 0.
	$$%\end{equation}
\end{enumerate}
\end{fact}

%By~\eqref{eq:CR_GC}, any accumulation point of the sequence generated by the CR method is a second-order critical point of~\eqref{opt:smooth_nc}. This contrasts with standard first-order algorithms, which typically (subsequentially) converge to first-order critical points. Moreover, this feature motivates us to explore the error bound conditions and other related regularity properties around the set of second-order critical points (see Section~\ref{sec:eb}).

We next review some existing results on the local convergence rate of the CR method. We start with the following result, which can be found in \cite[Theorem 4.1]{G81}.
%\begin{thm}\label{thm:CR_LC}
%Suppose that Assumption~\ref{ass:hess-lip} holds. If $\nabla^2 f(x^0) \succ 0$ and $\lambda_n(\nabla^2 f(x^0)) \geq 2\sqrt{L\|\nabla f(x^0)\|}$, the whole sequence $\{x^k\}$ generated by Algorithm \ref{alg:CR} converges at least R-quadratically to a point $\bar{x}$, which satisfies the second-order sufficient condition
%\begin{equation}
%\label{eq:sec-suff-cd}
%\nabla f(\bar{x}) = 0, \quad \nabla^2 f(\bar{x}) \succ 0.
%\end{equation}
%\end{thm}
%
%\noindent 
%Combining Theorem~\ref{thm:CR_GC} and Theorem~\ref{thm:CR_LC}, the following result is immediate.
\begin{fact}
\label{fact:CR_LC}
Suppose that Assumption~\ref{ass:hess-lip} holds. Let $\{x^k\}_{k\ge0}$ be the sequence generated by Algorithm~\ref{alg:CR} for solving problem~\eqref{opt:smooth_nc}. If an accumulation point $\bar{x}$ of $\{x^k\}_{k\ge0}$ satisfies
\begin{equation}
\label{eq:sec-suff-cd}
\nabla f(\bar{x}) = 0, \quad \nabla^2 f(\bar{x}) \succ 0,
\end{equation}
then the whole sequence $\{x_k\}_{k\ge0}$ converges at least Q-quadratically to $\bar{x}$.
\end{fact}
%\begin{proof}
%Let $\bar{x}$ be an accumulation point of $\{x^k\}$ satisfying~\eqref{eq:sec-suff-cd}. By the continuity of $\nabla f$ and $\nabla^2 f$, there exists a scalar $\delta>0$ such that
%\begin{equation}
%\label{eq:x-nonsing-neighbor}
%\nabla^2 f(x) \succ 0, \quad \lambda_n(\nabla^2 f(x)) \geq 2\sqrt{L\|\nabla f(x)\|}, \quad \forall x\in \mathbb{B}(\bar{x};\delta).
%\end{equation}
%Since $\bar{x}$ is an accumulation point of $\{x^k\}$, there exists $k_0\geq 0$ such that $x^{k_0}\in\mathbb{B}(\bar{x};\delta)$. By this, \eqref{eq:x-nonsing-neighbor} and Theorem~\ref{thm:CR_LC}, we can conclude that $\{x^k\}$ converges at least R-quadratically to $\bar{x}$.
%\end{proof}

As discussed in the Introduction, the non-degeneracy condition \eqref{eq:sec-suff-cd} implies that $\bar{x}$ is an isolated local minimizer of $f$, which does not hold in many applications. In an attempt to overcome such limitation, Nesterov and Polyak \cite{NP06} considered two classes of functions for which there can be non-isolated second-order critical points and showed that Algorithm \ref{alg:CR} converges superlinearly locally when applied to these functions. The first class is the so-called globally non-degenerate star-convex functions.
\begin{defi}
\label{defi:star-cvx}
We say that $f$ is star-convex if for any $x^*\in\mathcal{X}^*$,
\begin{equation}
\label{eq:def-star-cvx}
f(\alpha x^* + (1 - \alpha) x) \leq \alpha f^* + (1 - \alpha)f(x), \quad \forall x\in\RR^n, \; \forall\alpha\in[0,1].
\end{equation}
\end{defi}
%For example, $h(x) = \|x\|(1 - \exp(-\|x\|))$ and $h(x,y) = x^2y^2 + x^2 + y^2$ are star-convex functions. Star-convex functions are not necessarily convex, and it set of global minimizers need not be a singleton. 
\begin{defi}
\label{defi:global-nondeg}
We say that the optimal solution set $\mathcal{X}^*$ of $f$ is globally non-degenerate if there exists a scalar $\alpha>0$ such that 
\begin{equation} \label{eq:global-nondeg}
f(x) - f^* \geq \frac{\alpha}{2}\cdot{\rm dist}^2(x,\mathcal{X}^*), \quad \forall x\in\mathbb{R}^n.
\end{equation}
\end{defi}

\begin{fact}[{\cite[Theorem 5]{NP06}}]
\label{fact:super-conv-star}
Suppose that Assumption \ref{ass:hess-lip} holds, $f$ is star-convex, and $\mathcal{X}^*$ is globally non-degenerate. Then, there exist a scalar $\gamma>0$ and an integer $k_0\geq 0$ such that
$$%\begin{equation} \label{eq:super-conv-star}
f(x^{k+1}) - f^* \leq \gamma \left( f(x^k) - f^* \right)^\frac{3}{2}, \quad \forall k\geq k_0.
$$%\end{equation}
\end{fact}

The second class of functions studied in \cite{NP06} is the so-called gradient-dominated functions.
\begin{defi}
\label{defi:grad-domi}
We say that $f$ is gradient-dominated of degree $2$ if there exists a scalar $\tau_f>0$ such that
\begin{equation}
\label{eq:grad-domi}
f(x) - f^* \leq \tau_f \|\nabla f(x)\|^2, \quad \forall x\in\mathbb{R}^n.
\end{equation}
\end{defi}
It is worth mentioning that the inequality \eqref{eq:grad-domi} is an instance of the \L{}ojasiewicz inequality, which has featured prominently in the convergence analysis of iterative methods; see, \eg,~\cite{LP17} and the references therein. Indeed, recall that $f$ is said to satisfy the \L{}ojasiewicz inequality with exponent $\theta\in[\frac{1}{2},1)$ at $\bar{x}\in\mathbb{R}^n$ if there exist a scalar $c>0$ and a neighborhood $\mathcal{U}$ of $\bar{x}$ such that
$$ |f(x) - f(\bar{x})|^{\theta} \leq c \|\nabla f(x)\|, \quad \forall x\in\mathcal{U}. $$
Hence, the inequality~\eqref{eq:grad-domi} is simply the \L{}ojasiewicz inequality at any global minimizer of $f$ with $\theta = \frac{1}{2}$ and $\mathcal{U} = \mathbb{R}^n$.

\begin{fact}[{\cite[Theorem 7]{NP06}}]
\label{fact:super-conv-gd}
Suppose that Assumption \ref{ass:hess-lip} holds and $f$ is gradient-dominated of degree $2$. Then, there exist a scalar $\gamma>0$ and an integer $k_0\geq 0$ such that
$$%\begin{equation} \label{eq:super-conv-gd}
f(x^{k+1}) - f^* \leq \gamma \left( f(x^k) - f^* \right)^{\frac{4}{3}}, \quad \forall k\geq k_0.
$$%\end{equation}
\end{fact}

%It is worth mentioning that \eqref{eq:global-nondeg} and \eqref{eq:grad-domi} are global versions of the so-called quadratic growth inequality and \L{}ojasiewicz inequality, respectively.
From the definitions, it is not hard to see that both globally non-degenerate star-convex functions and gradient-dominated functions can be nonconvex and can have non-isolated second-order critical points. Nevertheless, the convergence rates obtained in Facts \ref{fact:super-conv-star} and \ref{fact:super-conv-gd} are weaker than that in Fact \ref{fact:CR_LC} in the following two aspects: (i) only superlinear rates of order $\frac{3}{2}$ and $\frac{4}{3}$ are established for these two classes respectively, while a quadratic rate is achieved in Fact \ref{fact:CR_LC}; (ii) only the  convergence rate of the objective values $\{f(x^k)\}_{k\ge0}$ is proved for these two classes, which is weaker than the convergence rate of the iterates $\{x^k\}_{k\ge0}$ in Fact \ref{fact:CR_LC}. As we shall see in Section \ref{sec:LC}, using our analysis approach, the superlinear convergence rates of $\{f(x^k)\}_{k\ge0}$ in Facts \ref{fact:super-conv-star} and \ref{fact:super-conv-gd} can be improved to the quadratic convergence rate of $\{x^k\}_{k\ge0}$.

\section{Error Bound for the Set of Second-Order Critical Points}\label{sec:eb}
Recall that $\mathcal{X}$ is the set of second-order critical points of $f$, which is a closed subset of $\mathbb{R}^n$ and assumed to be non-empty. In this section, we are interested in the local error bound (EB) condition (5) for $\mathcal{X}$, which we repeat here for the convenience of the readers.

\begin{ass}[EB Condition]
\label{ass:eb}
There exist scalars $\kappa,\rho > 0$ such that
\begin{equation}
\label{eq:eb}
\dist(x, \mathcal{X}) \leq \kappa\|\nabla f(x)\|, \quad\forall x\in\mathcal{N}(\mathcal{X};\rho).
\end{equation}
\end{ass}

Assumption \ref{ass:eb} is much weaker than the non-degeneracy assumption \eqref{eq:sec-suff-cd}. Indeed, if $\bar{x}\in\mathcal{X}$ satisfies condition \eqref{eq:sec-suff-cd}, then it is routine to show that $\bar{x}$ is an isolated second-order critical point and there exist scalars $\kappa,\rho>0$ such that $\dist(x,\mathcal{X})\leq \kappa\|\nabla f(x)\|$ whenever $\|x - \bar{x}\|\leq \rho$. On the other hand, the EB condition \eqref{eq:eb} can still be satisfied when $f$ has no isolated second-order critical points. For instance, it is not hard to verify that $f(x) = (\|x\|^2-1)^2$, whose set of second-order critical points is $\mathcal{X} = \{x: \|x\|=1\}$, satisfies the EB condition \eqref{eq:eb}. Furthermore, at the end of this section we shall show that both the globally non-degenerate star-convex functions and the gradient-dominated functions considered in Facts \ref{fact:super-conv-star} and \ref{fact:super-conv-gd} satisfy Assumption \ref{ass:eb}. In Section \ref{sec:app} we shall show that certain nonconvex functions that arise in phase retrieval and low-rank matrix recovery satisfy Assumption \ref{ass:eb} with overwhelming probability. 

%A notable feature of our EB condition \eqref{eq:eb} is that its target set is the set of second-order critical points. This contrasts with ordinary EB conditions in the literature, where $\mathcal{X}$ in \eqref{eq:eb} is typically replaced by the set of first-order critical points (\eg, see~\cite{LT93,F02,ZS15}). Such feature makes our EB condition especially useful for analyzing local convergence of iterative algorithms that are guaranteed to cluster at second-order critical points, such as the CR method of our interest. Moreover, some functions that does not satisfy ordinary EB conditions may still satisfy Assumption \ref{ass:eb}. For instance, consider the function $f(x) = (\|x\|^3 - 1)^2$, whose set of second-order critical points is $\mathcal{X} = \{x: \|x\|=1\}$ and set of first-order critical points is $\mathcal{X}^\prime = \mathcal{X}\cup\{0\}$. It is not hard to observe that $f$ satisfies Assumption \ref{ass:eb} while it does not satisfy EB condition for the set of first-order critical points.

In what follows, we prove that under some mild assumptions, the EB condition \eqref{eq:eb} is equivalent to a quadratic growth condition. For any $x\in\mathbb{R}^n$, we denote by $\hat{x}$ a projection of $x$ onto $\mathcal{X}$; \ie, $\hat{x} \in \Argmin_{z\in\mathcal{X}}\|x - z\|$.
%
%\begin{ass}[Proper seperation of isocost surfaces]
%\label{ass:iso-surf}
%There exists a scalar $\epsilon>0$ such that for any $x,y\in\mathbb{R}^n$ satisfying $x \in \mathcal{X}$, $y\in\mathcal{X}$, and $f(x) \neq f(y)$, it holds that $\|x - y\|\geq \epsilon$.
%\end{ass}

\begin{thm}
\label{thm:eb=qg}
Suppose that $\nabla^2 f(x)$ is uniformly continuous on $\mathcal{N}(\mathcal{X};\gamma)$ for some $\gamma>0$. Also, suppose that $f$ satisfies the following separation property: there exists an $\epsilon>0$ such that $\|x - y\|\geq \epsilon$ for any $x,y\in\mathcal{X}$ with $f(x)\neq f(y)$. Then, the following statements are equivalent.
\begin{enumerate}
	\item[(i)] There exist scalars $\kappa,\rho>0$ such that
	\begin{equation}
	\label{eq:eb-copy}
	\dist(x,\mathcal{X}) \leq \kappa\|\nabla f(x)\|, \quad \forall x\in\mathcal{N}(\mathcal{X};\rho).
	\end{equation}
	\item[(ii)] There exist scalars $\alpha,\beta>0$ such that
	\begin{equation}
	\label{eq:quad-growth}
	f(x) \geq f(\hat{x}) + \frac{\alpha}{2}\cdot\dist^2(x,\mathcal{X}), \quad \forall x\in \mathcal{N}(\mathcal{X};\beta).
	\end{equation}
\end{enumerate}
\end{thm}
Before presenting the proof, some remarks on the assumptions in Theorem \ref{thm:eb=qg} are in order. First, the uniform continuity of $\nabla^2 f(x)$ on $\mathcal{N}(\mathcal{X};\gamma)$ for some $\gamma>0$ holds if $\mathcal{X}$ is a compact set. Second, the separation property in Theorem \ref{thm:eb=qg} has appeared in \cite{LT93}, in which it was referred to as \emph{proper separation of isocost surfaces}, and has long played a role in the study of error bounds. It holds for many nonconvex functions in applications and holds trivially if $f$ is convex.

\begin{proof}[Proof of Theorem \ref{thm:eb=qg}]
We first prove $(i)\Rightarrow(ii)$. Suppose that~\eqref{eq:eb-copy} holds with some $\kappa,\rho>0$. Since $\nabla^2 f(x)$ is uniformly continuous on $\mathcal{N}(\mathcal{X};\gamma)$, there exists a scalar $\beta_0>0$ such that 
\begin{equation}
\label{eq:uniform-cont}
\|\nabla^2 f(x) - \nabla^2 f(y)\| \leq \frac{1}{4\kappa}, \quad \forall x,y\in\mathcal{N}(\mathcal{X};\gamma) \;\mbox{with}\; \|x - y\|\leq \beta_0.
\end{equation}
Let $\beta_1 := \min\{\beta_0,\rho,\gamma\}>0$, $x\in\mathcal{N}(\mathcal{X};\beta_1)$ be arbitrarily chosen, and $x(t) = \hat{x} + t(x - \hat{x})$ for $t\in[0,1]$. Thus, $\|x(t) - \hat{x}\|\leq \|x - \hat{x}\| \leq \beta_1$ for any $t\in[0,1]$. By~\eqref{eq:uniform-cont}, we have
$$ \|\nabla^2 f(x(t)) - \nabla^2 f(\hat{x})\| \leq \frac{1}{4\kappa}, \quad \forall t\in[0,1]. $$
This, together with the inequality $|\lambda_{\min}(A) - \lambda_{\min}(B)| \leq \|A - B\|$ for any real symmetric matrices $A$ and $B$ (see, \eg,~\cite[Corollary III.2.6]{B97}), yields
\begin{equation}
\label{eq:min-eig-lower-bd}
\lambda_{\min}[\nabla^2 f(x(t))] \geq \lambda_{\min}[\nabla^2 f(\hat{x})] - \frac{1}{4\kappa} \geq -\frac{1}{4\kappa}, \quad \forall t\in[0,1],
\end{equation}
where the last inequality is due to $\nabla^2 f(\hat{x})\succeq 0$. By the integral form of Taylor's series, we have
\begin{equation*}
\label{eq:taylor-series}
f(x) -  f(\hat{x}) = \langle \nabla f(\hat{x}), x - \hat{x}\rangle + \int_0^1(1-t)(x - \hat{x})^T\nabla^2 f(x(t)) (x - \hat{x})dt. 
\end{equation*}
This, together with~\eqref{eq:min-eig-lower-bd}, $\nabla f(\hat{x}) = 0$, and $\|x - \hat{x}\| = \dist(x,\mathcal{X})$, yields
%\begin{equation}
%\label{eq:qf-upper-bd}
%\begin{aligned}
%f(\hat{x}) - f(x) \leq \frac{1}{8\kappa}\cdot\|x - \hat{x}\|^2,
%\end{aligned}
%\end{equation}
%where the second inequality uses $\|x(t) - \hat{x}\| \leq \|x - \hat{x}\| \leq \beta_0$ and~\eqref{eq:uniform-cont}. Thus, it holds that
\begin{equation}
\label{eq:inverse-qg}
f(x) - f(\hat{x}) \geq - \frac{1}{8\kappa}\cdot\mbox{dist}^2(x,\mathcal{X}), \quad \forall x\in\mathcal{N}(\mathcal{X};\beta_1).
\end{equation}
Our next goal is to prove that there exists a scalar $\beta>0$ such that 
\begin{equation}
\label{eq:qg-claim}
f(x) \geq f(\hat{x}) + \frac{1}{16\kappa}\cdot\mbox{dist}^2(x,\mathcal{X}), \quad \forall x\in\mathcal{N}(\mathcal{X};\beta).
\end{equation}
This would then imply that statement (ii) holds. Suppose that~\eqref{eq:qg-claim} does not hold for any $\beta>0$. Then, there exist a sequence $\{x^k\}_{k\ge0}$ and a sequence of positive scalars $\{t_k\}_{k\ge0}$ such that $\lim_{k\rightarrow\infty}\mbox{dist}(x^k,\mathcal{X}) = 0$ and
\begin{equation}
\label{eq:contra-qg}
f(x^k) \leq f(\hat{x}^k) + \frac{1}{16\kappa}\cdot\mbox{dist}^2(x^k,\mathcal{X}) - t_k, \quad \forall k\geq 0.
\end{equation}
Without loss of generality, we assume that $x^k\in\mathcal{N}(\mathcal{X};\beta_1)$ for all $k\geq 0$. By \eqref{eq:contra-qg}, we have $x^k\notin\mathcal{X}$ for all $k\geq 0$. Let $\lambda_k:= \frac{1}{2}\cdot\dist(x^k,\mathcal{X})$. Hence, $\lim_{k\rightarrow\infty}\lambda_k = 0$ and $\lambda_k>0$ for all $k\geq 0$. Given any $k\geq 0$, consider the problem
\begin{equation}
\label{eq:auxilary-prob}
\begin{aligned}
v_k := \; & \min \; \left\{ f(x) + \frac{1}{8\kappa}\cdot\dist^2(x,\mathcal{X}) \right\} \\
& \;\;\mbox{s.t.} \,\, x\in\mathcal{N}(\mathcal{X};\beta_1)\cap\mathbb{B}\left(\hat{x}^k;\frac{\epsilon}{3}\right).
\end{aligned}
\end{equation}
Since $\hat{x}^k$ is feasible for~\eqref{eq:auxilary-prob} and $\hat{x}^k\in\mathcal{X}$, we have $v_k\leq f(\hat{x}^k)$. Let $x$ be an arbitrary feasible point of~\eqref{eq:auxilary-prob}. Then, it follows from~\eqref{eq:inverse-qg} that $ f(x)  + \frac{1}{8\kappa}\cdot\mbox{dist}^2(x,\mathcal{X}) \geq f(\hat{x})$. In addition, since $x\in\mathbb{B}\left( \hat{x}^k;\frac{\epsilon}{3} \right)$, we have
$\|\hat{x} - \hat{x}^k\| \leq \|x - \hat{x}\| + \|x - \hat{x}^k\| \leq 2\|x - \hat{x}^k\| \leq \frac{2}{3}\epsilon<\epsilon$.
This, together with the fact that $\hat{x},\hat{x}^k\in\mathcal{X}$ and our assumption in Theorem \ref{thm:eb=qg}, implies that $f(\hat{x}) = f(\hat{x}^k)$. Hence, every feasible point $x$ of \eqref{eq:auxilary-prob} satisfies $ f(x) + \frac{1}{8\kappa}\cdot\mbox{dist}^2(x,\mathcal{X}) \geq f(\hat{x}^k)$, 
which implies that $v_k\geq f(\hat{x}^k)$. Thus, we  can conclude that $v_k = f(\hat{x}^k)$. Combining this with~\eqref{eq:contra-qg}, we obtain
\begin{equation}
\label{eq:ekld-vp-1}
f(x^k) + \frac{1}{8\kappa}\cdot\mbox{dist}^2(x^k,\mathcal{X}) \leq v_k + \tau_k, \quad \forall k\geq 0,
\end{equation}
where $\tau_k = \frac{3}{16\kappa}\cdot\mbox{dist}^2(x^k,\mathcal{X}) - t_k$. Since $\lim_{k\rightarrow\infty}\mbox{dist}(x^k,\mathcal{X}) = 0$, there exists a $k_0\geq 0$ such that $x^k$ is feasible for~\eqref{eq:auxilary-prob} for any $k\geq k_0$. By this,~\eqref{eq:auxilary-prob},~\eqref{eq:ekld-vp-1}, and Ekeland's variational principle (see, \eg, \cite[Theorem 2.26]{M06}), there exists a sequence $\{z^k\}_{k\geq k_0}$ such that for all $k\geq k_0$, $\|x^k - z^k\|\leq \lambda_k$ and
\begin{equation}
\label{eq:ekld-vp-2}
\begin{aligned}
z^k = \;\argmin & \; \left\{ f(x) + \frac{1}{8\kappa}\cdot\dist^2(x,\mathcal{X}) + \frac{\tau_k}{\lambda_k}\|x - z^k\| \right\} \\
\mbox{s.t.}\;\; & \;\, x\in\mathcal{N}(\mathcal{X};\beta_1)\cap\mathbb{B}\left(\hat{x}^k;\frac{\epsilon}{3}\right).
\end{aligned}
\end{equation}
Since $\lim_{k\rightarrow\infty}\lambda_k = 0$, we have $\lim_{k\rightarrow\infty}\|x^k - z^k\| = 0$. In addition, noticing that
$$ \mbox{dist}(z^k,\mathcal{X}) \leq \|z^k - \hat{x}^k\| \leq \|z^k - x^k\| + \|x^k - \hat{x}^k\| = \|z^k - x^k\| + \mbox{dist}(x^k,\mathcal{X}), $$
we obtain $\lim_{k\rightarrow\infty}\|z^k - \hat{x}^k\| = \lim_{k\rightarrow\infty}\mbox{dist}(z^k,\mathcal{X}) = 0$. Hence, there exists a $k_1\geq k_0$ such that $z^k$ is in the interior of the feasible set of~\eqref{eq:ekld-vp-2} for all $k\geq k_1$. Consequently, by the generalized Fermat's rule (see, \eg, \cite[Theorem 10.1]{RW98}), we have
\begin{equation}
\label{eq:opt-z-int-all}
0 \in \partial\left( f(\cdot) + \frac{1}{8\kappa}\cdot\mbox{dist}^2(\cdot,\mathcal{X}) + \frac{\tau_k}{\lambda_k}\|\cdot - z^k\| \right)(z^k), \footnote[3]{Given an extended-real-valued function $h:\mathbb{R}^n\rightarrow(-\infty,+\infty]$ and an $x\in\mbox{dom}(h):=\{z\in\mathbb{R}^n: h(z)<\infty\}$, we denote by $\partial h(x)$ the \emph{limiting subdifferential} (known also as the general or Mordukhovich subdifferential) of $h$ at $x$; see, \eg,~\cite[Definition 1.77]{M06}.} \quad \forall k\geq k_1.
\end{equation}
Since $f$ is continuously differentiable, we obtain from \cite[Corollary 1.82]{M06} that $\partial f(z^k) = \{\nabla f(z^k)\}$. In addition, we have
$$%\begin{equation} \label{eq:subdiff-dist-2}
\partial\left(\mbox{dist}^2(\cdot,\mathcal{X})\right)(z^k) = 2\cdot\mbox{dist}(z^k,\mathcal{X})\cdot\partial\left(\mbox{dist}(\cdot,\mathcal{X})\right)(z^k) \subset 2\cdot\mbox{dist}(z^k,\mathcal{X})\cdot \mathbb{B}(0;1),
$$%\end{equation}
where the equality follows from \cite[Corollary 1.111(i)]{M06} and the inclusion is due to \cite[Example 8.53]{RW98}. Also, we have $\partial\left(\|\cdot - z^k\|\right)(z^k) = \mathbb{B}(0;1)$. These, together with \eqref{eq:opt-z-int-all}, yield
\begin{align}
0 & \in \partial\left( f(\cdot) + \frac{1}{8\kappa}\cdot\mbox{dist}^2(\cdot,\mathcal{X}) + \frac{\tau_k}{\lambda_k}\|\cdot - z^k\| \right)(z^k) \nonumber \\
& = \nabla f(z^k) + \partial \left(\frac{1}{8\kappa}\cdot\mbox{dist}^2(\cdot,\mathcal{X}) + \frac{\tau_k}{\lambda_k}\|\cdot - z^k\| \right)(z^k) \label{eq:subdiff-=} \\
& \subset \nabla f(z^k) + \partial\left(\frac{1}{8\kappa}\cdot\mbox{dist}^2(\cdot,\mathcal{X})\right)(z^k) + \partial \left(\frac{\tau_k}{\lambda_k}\|\cdot - z^k\| \right)(z^k) \label{eq:subdiff-in} \\
& \subset \nabla f(z^k) + \left(\frac{1}{4\kappa}\cdot\mbox{dist}(z^k,\mathcal{X}) + \frac{\tau_k}{\lambda_k}\right)\mathbb{B}(0;1), \quad \forall k\geq k_1,\label{eq:opt-z-int}
\end{align}
where \eqref{eq:subdiff-=} and \eqref{eq:subdiff-in} are due to \cite[Exercise 10.10]{RW98}. By \eqref{eq:opt-z-int}, we have 
\begin{equation}
\label{eq:norm-bd}
\|\nabla f(z^k)\| \leq \frac{1}{4\kappa}\cdot\mbox{dist}(z^k,\mathcal{X}) + \frac{\tau_k}{\lambda_k}, \quad \forall k\geq k_1.
\end{equation}
Moreover, we have $z^k\in\mathcal{N}(\mathcal{X};\beta_1)$ for all $k\geq k_0$ from \eqref{eq:ekld-vp-2}. This, together with $\beta_1\leq\rho$ and~\eqref{eq:eb-copy}, yields $\mbox{dist}(z^k,\mathcal{X}) \leq \kappa\|\nabla f(z^k)\|$ for all $k\geq k_0$. By this, $k_1\geq k_0$, and \eqref{eq:norm-bd}, we have
$$ \mbox{dist}(z^k,\mathcal{X}) \leq \kappa\|\nabla f(z^k)\| \leq \frac{1}{4}\cdot\mbox{dist}(z^k,\mathcal{X}) + \frac{\kappa\tau_k}{\lambda_k}, \quad \forall k\geq k_1,$$
which results in $\mbox{dist}(z^k,\mathcal{X}) \leq \frac{4\kappa\tau_k}{3\lambda_k}$ for all $k\geq k_1$. This further leads to
$$ \mbox{dist}(x^k,\mathcal{X}) = \|x^k - \hat{x}^k\| \leq \|x^k - \hat{z}^k\| \leq \|x^k - z^k\| + \mbox{dist}(z^k,\mathcal{X}) \leq \lambda_k + \frac{4\kappa\tau_k}{3\lambda_k}, \quad \forall k\geq k_1.$$
By the definitions of $\tau_k$ and $\lambda_k$, the above yields
$$ \mbox{dist}^2(x^k,\mathcal{X}) \leq \mbox{dist}^2(x^k,\mathcal{X}) - \frac{4\kappa t_k}{3}, \quad \forall k\geq k_1,$$
which is a contradiction since $\kappa>0$ and $t_k >0$ for all $k\geq 0$. Therefore, there exists a scalar $\beta>0$ such that~\eqref{eq:qg-claim} holds, which implies that statement (ii) holds.

\smallskip
We next prove $(ii)\Rightarrow(i)$. Suppose that~\eqref{eq:quad-growth} holds with some $\alpha,\beta>0$. Since $\nabla^2 f(x)$ is uniformly continuous on $\mathcal{N}(\mathcal{X};\gamma)$, there exists a scalar $\rho_0>0$ such that 
\begin{equation}
\label{eq:uniform-cont-2}
\|\nabla^2 f(x) - \nabla^2 f(y)\| \leq \frac{\alpha}{2}, \quad \forall x,y\in\mathcal{N}(\mathcal{X};\gamma) \; \mbox{with}\; \|x - y\|\leq \rho_0.
\end{equation}
Let $\rho_1 = \min\{\rho_0,\beta,\gamma\}>0$, $x\in\mathcal{N}(\mathcal{X};\rho_1)$ be arbitrarily chosen, and $\tilde{x}(t) = x + t(\hat{x} - x)$ for $t\in[0,1]$. Using the same arguments as those for~\eqref{eq:min-eig-lower-bd}, one has
\begin{equation}
\label{eq:min-eig-lower-bd-2}
\lambda_{\min}[\nabla^2 f(\tilde{x}(t))] \geq -\frac{\alpha}{2}, \quad \forall t\in[0,1]. 
\end{equation}
By~\eqref{eq:quad-growth}, \eqref{eq:min-eig-lower-bd-2}, and the integral form of Taylor's series, we obtain
\begin{equation*}
\begin{aligned}
\langle \nabla f(x), x - \hat{x}\rangle & = f(x) - f(\hat{x}) + \int_0^1(1-t)(\hat{x} - x)^T\nabla^2 f(\tilde{x}(t))(\hat{x} - x)dt \\
& \geq \frac{\alpha}{2}\|x - \hat{x}\|^2 - \frac{\alpha}{4}\|x - \hat{x}\|^2 = \frac{\alpha}{4}\|x - \hat{x}\|^2.
\end{aligned}
\end{equation*}
Applying the Cauchy-Schwarz inequality and using $\mbox{dist}(x,\mathcal{X}) = \|x - \hat{x}\|$, the above yields
$$ \mbox{dist}(x,\mathcal{X}) \leq \frac{4}{\alpha}\|\nabla f(x)\|, \quad \forall x\in\mathcal{N}(\mathcal{X};\rho_1). $$
Therefore, statement (i) holds as well.% The proof is then completed.
\end{proof}

\smallskip
{\bf Remark.} When $f$ is convex, $\mathcal{X}$ reduces to the set of optimal solutions to $f$ and it is known that the EB condition \eqref{eq:eb-copy} is equivalent to the quadratic growth condition \eqref{eq:quad-growth}; see, \eg, \cite{AG08}. When $f$ is nonconvex, Drusvyatskiy {\it et al.}~\cite{DMN14} studied these two regularity conditions for the set of first-order critical points (replacing $\mathcal{X}$ in both \eqref{eq:eb-copy} and \eqref{eq:quad-growth} by the set of first-order critical points) and proved that they are equivalent under an additional quadratic decrease condition; see \cite[Theorem~3.1]{DMN14}. Our Theorem \ref{thm:eb=qg} is motivated by \cite[Theorem 3.1]{DMN14} and shows that for the set of second-order critical points of a twice continuously differentiable function, the EB condition \eqref{eq:eb-copy} and the quadratic growth condition \eqref{eq:quad-growth} are equivalent without requiring the said additional condition.

\begin{coro}
Suppose that Assumption \ref{ass:eb} and the premise of Theorem \ref{thm:eb=qg} hold. Then, any second-order critical point of $f$ is a local minimizer.
\end{coro}
\begin{proof}
Let $\bar{x}$ be an arbitrary second-order critical point of $f$. By Theorem \ref{thm:eb=qg} and Assumption \ref{ass:eb}, the quadratic growth condition \eqref{eq:quad-growth} holds for some $\alpha,\beta>0$. Let $\delta = \min\{\beta,\frac{\epsilon}{3}\}$ and $x$ be an arbitrary point in $\mathcal{N}(\bar{x};\delta)$. It then follows from \eqref{eq:quad-growth} that $f(x) \geq f(\hat{x})$. Moreover, it holds that $\|\hat{x} - \bar{x}\| \leq \|x - \hat{x}\| + \|x - \bar{x}\| \leq 2\|x - \bar{x}\| \leq \frac{2}{3}\epsilon < \epsilon$. By this and the separation property in Theorem \ref{thm:eb=qg}, we have $f(\hat{x}) = f(\bar{x})$. Hence, we obtain $f(x) \geq f(\bar{x})$ for all $x\in\mathcal{N}(\bar{x};\delta)$, which implies that $\bar{x}$ is a local minimizer of $f$. 
\end{proof}

\smallskip
For the rest of this section, we show that the classes of functions considered in Facts \ref{fact:super-conv-star} and~\ref{fact:super-conv-gd} satisfy Assumption \ref{ass:eb}.
\begin{prop}
\label{prop:eb-star}
Suppose that $f$ is star-convex, $\mathcal{X}^*$ is globally non-degenerate, and $\nabla^2 f(x)$ is uniformly continuous on $\mathcal{N}(\mathcal{X}^*;\gamma)$ for some $\gamma>0$. Then, $f$ satisfies Assumption~\ref{ass:eb}.
\end{prop}
\begin{proof}
We first show that for star-convex functions, the set of second-order critical points equals the set of optimal solutions; \ie, $\mathcal{X} = \mathcal{X}^*$. Since it is clear that $\mathcal{X}^*\subset\mathcal{X}$, it suffices to show that $\mathcal{X}\subset\mathcal{X}^*$. Suppose on the contrary that $x\notin\mathcal{X}^*$ for some $x\in\mathcal{X}$. Hence, $\nabla f(x)=0$ and $f(x) > f(x^*)$ for any $x^*\in\mathcal{X}^*$. By this and \eqref{eq:def-star-cvx}, we have that for any $x^*\in\mathcal{X}$,
\begin{align*}
\langle \nabla f(x), x^* - x\rangle & = \lim_{\alpha\downarrow 0} \frac{ f(x + \alpha(x^*-x)) - f(x)}{\alpha} \\
& \leq \lim_{\alpha\downarrow 0} \frac{ \alpha f(x^*) + (1-\alpha) f(x) - f(x) }{\alpha} = f(x^*) - f(x) < 0,
\end{align*}
which contradicts with $\nabla f(x)=0$. Hence, we obtain $\mathcal{X} = \mathcal{X}^*$. This, together with our assumption in Proposition \ref{prop:eb-star}, implies that $\nabla^2 f(x)$ is uniformly continuous on $\mathcal{N}(\mathcal{X};\gamma)$ for some $\gamma>0$. Also, since $\mathcal{X} = \mathcal{X}^*$, we have $f(x) = f(y) = f^*$ for any $x,y\in\mathcal{X}$, which implies that the separation property in Theorem \ref{thm:eb=qg} holds. Moreover, by $\mathcal{X} = \mathcal{X}^*$ and the assumption that $\mathcal{X}^*$ is globally non-degenerate, statement (ii) of Theorem \ref{thm:eb=qg} holds. Hence, statement (i) of Theorem~\ref{thm:eb=qg} holds as well, which implies that $f$ satisfies Assumption \ref{ass:eb}. 
\end{proof}

\begin{prop}
\label{prop:eb-grad-domi}
Suppose that $f$ is gradient-dominated of degree $2$ and $\nabla^2 f(x)$ is uniformly continuous on $\mathcal{N}(\mathcal{X}^*;\gamma)$ for some $\gamma>0$. Then, $f$ satisfies Assumption \ref{ass:eb}.
\end{prop}
\begin{proof}
Due to \eqref{eq:grad-domi}, one can see that for any $x\notin\mathcal{X}^*$, we have $\nabla f(x)\neq 0$, which immediately implies that $\mathcal{X} \subset \mathcal{X}^*$. This, together with $\mathcal{X}^*\subset\mathcal{X}$, yields $\mathcal{X} = \mathcal{X}^*$. It then follows from the same arguments as those in the proof of Proposition \ref{prop:eb-star} that the premise of Theorem \ref{thm:eb=qg} holds. Our next goal is to prove 
\begin{equation}
\label{eq:quad-grow}
f(x) - f^* \geq \frac{1}{4\tau_f}\cdot\dist^2(x,\mathcal{X}^*), \quad \forall x\in\mathbb{R}^n.
\end{equation} 
Notice that \eqref{eq:quad-grow} holds trivially for $x\in\mathcal{X}^*$. Let $\tilde{x}\in\mathbb{R}^n\setminus\mathcal{X}^*$ be arbitrarily chosen. Consider the differential equation
\begin{equation}
\label{eq:diff-eq}
\left\{
\begin{array}{l}
u(0) = \tilde{x}, \\
\dot{u}(t) = -\nabla f(u(t)), \quad \forall t>0.
\end{array}\right.
\end{equation}
Since $\nabla f$ is continuously differentiable on $\mathbb{R}^n$, it is Lipschitz continuous on any compact subset of $\mathbb{R}^n$. It then follows from the Picard-Lindel\"{o}f Theorem (see, \eg, \cite[Theorem II.1.1]{H02}) that there exists a $\delta>0$ such that \eqref{eq:diff-eq} has a unique solution $u_{\tilde{x}}(t)$ for $t\in[0,\delta]$. Let $[0,\nu)$ be the maximal interval of existence for $u_{\tilde{x}}(t)$, where $\nu\leq \infty$.\footnote{An interval $[0,\nu)$ is called a \emph{maximal interval} of existence for $u_{\tilde{x}}(t)$ if there does not exist an extension $\tilde{u}_{\tilde{x}}(t)$ of $u_{\tilde{x}}(t)$ over an interval $[0,\tilde{\nu})$ such that $\tilde{u}_{\tilde{x}}(t)$ remains a solution to \eqref{eq:diff-eq} and $\tilde{\nu}>\nu$; see, \eg, \cite[p.~12]{H02}.}  
 Define $H(t) := f(u_{\tilde{x}}(t)) - f^*$ for $t\in[0,\nu)$. Then, we have
\begin{equation}
\label{eq:neg-deri}
\dot{H}(t) = \langle \nabla f(u_{\tilde{x}}(t)), \dot{u}_{\tilde{x}}(t) \rangle = -\|\nabla f(u_{\tilde{x}}(t))\| \cdot \|\dot{u}_{\tilde{x}}(t)\|, \quad \forall t\in[0,\nu),
\end{equation} 
where the second equality is due to \eqref{eq:diff-eq}. Using~\eqref{eq:grad-domi} and the definition of $H$, we get
\begin{equation*}
\label{eq:neg-deri-1}
\dot{H}(t) \leq -\frac{H(t)^{\frac{1}{2}}}{\sqrt{\tau_f}}\|\dot{u}_{\tilde{x}}(t)\|, \quad \forall t\in[0,\nu).
\end{equation*} 
Recall that $\nabla f(x) = 0$ for any $x\in\mathcal{X}^*$. This implies that there does not exist a $\bar{t}\in[0,\nu)$ such that $u_{\tilde{x}}(\bar{t})\in\mathcal{X}^*$, for otherwise $u_{\tilde{x}} \equiv u_{\tilde{x}}(\bar{t})$ is the unique solution to \eqref{eq:diff-eq}, which contradicts with $u_{\tilde{x}}(0) = \tilde{x}\notin\mathcal{X}^*$.
Hence, $H(t)> 0$ for all $t\in[0,\nu)$ and %. This, together with the above inequality, yields
\begin{equation}
\label{eq:ub-ut}
\|\dot{u}_{\tilde{x}}(t)\|\leq - \sqrt{\tau_f}\frac{\dot{H}(t)}{H(t)^{\frac{1}{2}}}  = -2\sqrt{\tau_f}\left[H(t)^{\frac{1}{2}}\right]^\prime.
\end{equation} 
Then, for any $0\leq s_1<s_2<\nu$, we have
\begin{align}
\|u_{\tilde{x}}(s_2) - u_{\tilde{x}}(s_1)\| & = \left\|\int_{s_1}^{s_2} \dot{u}_{\tilde{x}}(t)dt \right\| \leq \int_{s_1}^{s_2}\|\dot{u}_{\tilde{x}}(t)\|dt \nonumber \\
& \leq \int_{s_1}^{s_2} - 2\sqrt{\tau_f}\left[H(t)^{\frac{1}{2}}\right]^\prime dt \nonumber \\ %\label{eq:use-ub-ut} \\
& = 2\sqrt{\tau_f}\left[H(s_1)^{\frac{1}{2}} - H(s_2)^{\frac{1}{2}}\right] \label{eq:ub-fun}.
\end{align}
Substituting $s_1 = 0$ in \eqref{eq:ub-fun} and using $u_{\tilde{x}}(0) = \tilde{x}$ and $H(t)>0$ for any $t\in[0,\nu)$, we obtain 
\begin{equation}
\label{eq:limit-bd}
\|u_{\tilde{x}}(s_2) - \tilde{x}\|\leq 2\sqrt{\tau_f}H(0)^{\frac{1}{2}} = 2\sqrt{\tau_f}\left(f(\tilde{x}) - f^*\right)^{\frac{1}{2}}, \quad \forall s_2\in[0,\nu).
\end{equation}
Next, we claim that $\nu = \infty$. Suppose to the contrary that $\nu<\infty$.  Then, it follows from \cite[Corollary II.3.2]{H02} that $\|u_{\tilde{x}}(t)\|\rightarrow\infty$ as $t\nearrow\nu$. However, the above inequality implies that
$$ \|u_{\tilde{x}}(t)\| \leq \|\tilde{x}\| + \|u_{\tilde{x}}(t) - \tilde{x}\|\leq \|\tilde{x}\| + 2\sqrt{\tau_f}\left(f(\tilde{x}) - f^*\right)^{\frac{1}{2}}, \quad \forall t\in[0,\nu),$$
which yields a contradiction. Hence, the claim $\nu = \infty$ is true. In addition, we have $\dot{H}(t)\leq 0$ for all $t\in [0,\infty)$ from~\eqref{eq:neg-deri}, which implies that $H(t)$ is non-increasing on $[0,\infty)$. This, together with $H(t)>0$ for all $t\in[0,\infty)$, implies that $\lim_{t\rightarrow\infty}H(t)$ exists. It then follows from this and \eqref{eq:ub-fun} that $u_{\tilde{x}}(t)$ has the Cauchy property and hence $\lim_{t\rightarrow\infty}u_{\tilde{x}}(t)$ exists. Let $u_{\tilde{x}}(\infty) := \lim_{t\rightarrow\infty}u_{\tilde{x}}(t)$. We claim that $\nabla f(u_{\tilde{x}}(\infty)) = 0$. Indeed, if $\nabla f(u_{\tilde{x}}(\infty)) \neq 0$, then by \eqref{eq:diff-eq}, \eqref{eq:neg-deri}, and the continuity of $\nabla f$, we have
$$ \lim_{t\rightarrow\infty}\dot{H}(t) = -\lim_{t\rightarrow\infty}\|\nabla f(u_{\tilde{x}}(t))\|\cdot\|\dot{u}_{\tilde{x}}(t)\| = -\lim_{t\rightarrow\infty} \|\nabla f(u_{\tilde{x}}(t))\|^2 = -\|\nabla f(u_{\tilde{x}}(\infty))\|^2 < 0, $$
which contradicts with the fact that $\lim_{t\rightarrow\infty}H(t)$ exists. Hence, we have $\nabla f(u_{\tilde{x}}(\infty)) = 0$. This, together with \eqref{eq:grad-domi}, yields $f(u_{\tilde{x}}(\infty)) = f^*$ and hence $u_{\tilde{x}}(\infty)\in\mathcal{X}^*$. By~\eqref{eq:limit-bd}, this gives
$$ \mbox{dist}(\tilde{x},\mathcal{X}^*) \leq \|u_{\tilde{x}}(\infty) - \tilde{x}\| = \lim_{t\rightarrow\infty}\|u_{\tilde{x}}(t) - \tilde{x}\| \leq 2\sqrt{\tau_f}\left(f(\tilde{x}) - f^*\right)^{\frac{1}{2}}, $$
which implies that \eqref{eq:quad-grow} holds for $x = \tilde{x}$. Since $\tilde{x}\in\mathbb{R}^n\setminus\mathcal{X}^*$ is arbitrary, we conclude that \eqref{eq:quad-grow} holds for all $x\in\mathbb{R}^n$. By \eqref{eq:quad-grow} and the fact that $\mathcal{X} = \mathcal{X}^*$, statement (ii) of Theorem \ref{thm:eb=qg} holds. Hence, statement (i) of Theorem \ref{thm:eb=qg} holds as well, which implies that $f$ satisfies Assumption \ref{ass:eb}.
\end{proof}

\section{Quadratic Convergence of the CR Method}\label{sec:LC}
In this section, we establish the quadratic rate of convergence of the CR method under the local EB condition proposed in Section \ref{sec:eb}. 
To proceed, we start with the following consequence of Assumption~\ref{ass:hess-lip}.
\begin{fact}[{\cite[Lemma 1]{NP06}}]
\label{fact:lip-result}
Suppose that Assumption~\ref{ass:hess-lip} holds. Then, for any $x,y\in\mathcal{F}$,
\begin{equation}
\label{eq:lip-result-1}
\|\nabla f(y) - \nabla f(x) - \nabla^2 f(x)(y - x)\| \leq \frac{L}{2}\|y - x\|^2.
\end{equation}
%and 
%\begin{equation}
%\label{eq:lip-result-2}
%\left|f(y) - f(x) - \langle \nabla f(x), y - x \rangle - \frac{1}{2}\langle \nabla^2 f(x) (y - x), y - x \rangle \right| \leq \frac{L}{6}\|y - x\|^3.
%\end{equation}
\end{fact}

We next prove the following intermediate lemma.

\begin{lem}
\label{lem:step-bd}
Suppose that Assumption~\ref{ass:hess-lip} holds. Let $x\in\mathcal{F}$ and $\hat{x}$ be a projection point of $x$ to $\mathcal{X}$. If $\hat{x}\in\mathcal{F}$, then for any $\sigma>0$, we have
\begin{equation}
\label{eq:step-bd}
\|p_\sigma(x) - x\| \leq \left(1+\frac{L}{\sigma} + \sqrt{\left(1 + \frac{L}{\sigma}\right)^2 + \frac{L}{\sigma}}\right)\cdot\dist(x,\mathcal{X}).
\end{equation}
\end{lem}
\begin{proof}
For simplicity, we denote $x^+ := p_\sigma(x)$. By~\eqref{eq:def-p-sigma} and the first-order optimality condition of~\eqref{eq:true_cubic_model}, one has
\begin{equation}
\label{eq:first-opt-subprob}
0 = \nabla f(x) + \nabla^2 f(x)(x^+ - x) + \frac{\sigma}{2}\|x^+ - x\|(x^+ - x).
\end{equation}
Since $\hat{x}\in\mathcal{X}$, we have $\nabla f(\hat{x}) = 0$ and $\nabla^2 f(\hat{x}) \succeq 0$. By~\eqref{eq:first-opt-subprob} and $\nabla f(\hat{x}) = 0$, it is not hard to verify that
$$ 
\begin{aligned}
\left( \nabla^2 f(\hat{x}) + \frac{\sigma \|x^+ - x\|}{2}I_n\right) (x^+ - \hat{x}) &  = \nabla f(\hat{x}) - \nabla f(x) - \nabla^2 f(\hat{x})(\hat{x} - x) \\
& \quad - \frac{\sigma}{2}\|x^+ - x\|(\hat{x} - x) - \left( \nabla^2 f(x) - \nabla^2 f(\hat{x})\right)(x^+ - x). 
\end{aligned}
$$
Since $\nabla^2 f(\hat{x})\succeq 0$, we have
$$ \left\| \left( \nabla^2 f(\hat{x}) + \frac{\sigma \|x^+ - x\|}{2}I_n\right) (x^+ - \hat{x}) \right\| \geq \frac{\sigma}{2}\|x^+ - x\| \cdot \|x^+ - \hat{x}\|.$$
This, together with the above equality, yields
$$
\begin{aligned}
\frac{\sigma}{2}\|x^+ - x\| \cdot \|x^+ - \hat{x}\| & \leq \|\nabla f(\hat{x}) - \nabla f(x) - \nabla^2 f(\hat{x})(\hat{x} - x)\| + \frac{\sigma}{2}\|x^+ - x\| \cdot \|x - \hat{x}\| \\
& \quad + \| \nabla^2 f(x) - \nabla^2 f(\hat{x}) \| \cdot \|x^+ - x\|\\
& \leq \frac{L}{2}\|x - \hat{x}\|^2 + \left(\frac{\sigma}{2}+L\right)\|x^+ - x\| \cdot \|x - \hat{x}\|,
\end{aligned}
$$
where the second inequality is due to Fact~\ref{fact:lip-result} and the assumption that $\hat{x}\in\mathcal{F}$. Using the triangle inequality $\|x^+ - \hat{x}\| \geq \|x^+ - x\| - \|x - \hat{x}\|$, we further obtain
$$ \frac{\sigma}{2}\|x^+ - x\|^2 \leq \frac{L}{2}\|x - \hat{x}\|^2 + (\sigma + L)\|x^+ - x\|\cdot\|x - \hat{x}\|. $$
By solving the above   quadratic inequality, one has
$$ \|x^+ - x\| \leq \left(1+\frac{L}{\sigma} + \sqrt{\left(1 + \frac{L}{\sigma}\right)^2 + \frac{L}{\sigma}}\right) \cdot \|x - \hat{x}\|.$$
Noticing that $x^+ = p_{\sigma}(x)$ and $\mbox{dist}(x,\mathcal{X}) = \|x - \hat{x}\|$, we obtain the desired inequality~\eqref{eq:step-bd}.
\end{proof}

\smallskip
{\bf Remark.} As we shall see in the sequel, Lemma \ref{lem:step-bd} implies that there exists a $c_1>0$ such that 
\begin{equation}
\label{eq:conse-bd}
\|x^{k+1} - x^k\| \leq c_1\cdot\dist(x^k,\mathcal{X})
\end{equation}
for all sufficiently large $k$, where $\{x^k\}_{k\ge0}$ is the sequence generated by Algorithm \ref{alg:CR}. It is known that establishing \eqref{eq:conse-bd} is an important step for analyzing local convergence of Newton-type methods with non-isolated solutions. However, our proof of Lemma \ref{lem:step-bd} is novel. Indeed, in most cases \eqref{eq:conse-bd} is obtained based on the property that $x^{k+1}$ is the minimizer of a strongly convex function (see, \eg, \cite{LFQY04,YZS16,YF01,FY05,BM14}), which does not apply to the CR method. Moreover, in our proof of Lemma \ref{lem:step-bd}, the fact that $\nabla^2 f(x)\succeq 0$ for any $x\in\mathcal{X}$ plays a crucial role.

\smallskip
Now we are ready to present the main result of this section.

\begin{thm}
\label{thm:CR-LC-exact}
Suppose that Assumptions~\ref{ass:hess-lip} and \ref{ass:eb} hold. Let $\{x^k\}_{k\ge0}$ be the sequence of iterates generated by the CR method. If $\mathcal{L}(f(x^k))$ is bounded for some $k\geq 0$, then the whole sequence $\{x^k\}_{k\ge0}$ converges at least Q-quadratically to a point $x^*\in\mathcal{X}$.
\end{thm}
\begin{proof}
Let $\hat{x}^k$ be a projection point of $x^k$ to $\mathcal{X}$; \ie, $ \hat{x}^k \in \Argmin_{z\in\mathcal{X}}\|z - x^k\|$. Let $\bar{\mathcal{X}}$ be the set of accumulation points of $\{x^k\}_{k\ge0}$. By~\eqref{eq:suff-decrease-copy}, we have $f(x^{k+1})\leq \bar{f}_{\sigma_k}(x^k) \leq f(x^k)$ for all $k$. This, together with the boundedness of $\mathcal{L}(f(x^k))$ for some $k\geq 0$, implies the boundedness of $\{x^k\}_{k\ge0}$. Hence, $\bar{\mathcal{X}}$ is non-empty and bounded, and we have $\lim_{k\rightarrow\infty}\mbox{dist}(x^k,\bar{\mathcal{X}}) = 0$. By Fact~\ref{fact:CR_GC} (iii), we have $\bar{\mathcal{X}} \subset \mathcal{X}$. Thus, $\lim_{k\rightarrow\infty}\mbox{dist}(x^k,\mathcal{X}) = \lim_{k\rightarrow\infty}\|x^k - \hat{x}^k\| = 0$. It then follows from Assumption~\ref{ass:eb} that there exists a $k_1\geq 0$ such that
$$%\begin{equation} \label{eq:eb-iterates}
\mbox{dist}(x^k, \mathcal{X}) \leq \kappa\|\nabla f(x^k)\|, \quad \forall k\geq k_1.
$$%\end{equation}
In addition, since $x^k\in \mathcal{L}(f(x^0)) \subset\mbox{int}(\mathcal{F})$ for all $k\geq 0$ and $\{x^k\}_{k\ge0}$ is bounded, there exists a compact set $\mathcal{M} \subset {\rm int}(\mathcal{F})$ such that $\{x^k\}_{k\ge0} \subset \mathcal{M}$. Also, it follows from $\{x^k\}_{k\geq 0}\subset\mathcal{M}$ and $\lim_{k\rightarrow\infty}\|x^k - \hat{x}^k\| = 0$ that $\lim_{k\rightarrow\infty}\mbox{dist}(\hat{x}^k,\mathcal{M}) = 0$. This, together with $\mathcal{M}\subset\mbox{int}(\mathcal{F})$ and the compactness of $\mathcal{M}$, implies that $\hat{x}^k\in\mbox{int}(\mathcal{F})$ for all sufficiently large $k$. Hence, there exists a $k_2\geq 0$ such that
\begin{equation}
\label{eq:iterates-in-F}
x^k \in \mathcal{F}, \quad \hat{x}^k\in\mathcal{F}, \quad \forall k\geq k_2.
\end{equation}
\begin{comment}
In particular, we can find a finite collection of points $y^1,\ldots,y^N \in \mathcal{M}$ and scalars $\zeta_1,\ldots,\zeta_N>0$ such that $\mathcal{M} \subset \bigcup_{i=1}^N \mathbb{B}(y^i;\zeta_i) \subset {\rm int}(\mathcal{F})$. Now, for each $k\ge0$, let $i_k \in \{1,\ldots,N\}$ be such that $x^k \in \mathbb{B}(y^{i_k};\zeta_{i_k})$. Then, we have
$$ \| \hat{x}^k - y^{i_k} \| \le \| \hat{x}^k - x^k \| + \| x^k - y^{i_k} \| \le \dist(x^k,\mathcal{X}) + \zeta_{i_k}. $$
Since $\lim_{k\rightarrow\infty}\mbox{dist}(x^k,\mathcal{X}) = 0$ and ${\rm int}(\mathcal{F})$ is open, the above inequality and the finiteness of the cover $\bigcup_{i=1}^N \mathbb{B}(y^i;\zeta_i)$ imply that there exists a $k_2\geq 0$ such that
\begin{equation}
\label{eq:iterates-in-F}
x^k \in \mathcal{F}, \quad \hat{x}^k\in\mathcal{F}, \quad \forall k\geq k_2.
\end{equation}
\end{comment}
Hence, for any $k\geq \bar{k} := \max\{k_1,k_2\}$, we have
\begin{equation}
\label{eq:derive-quad-conv}
\begin{aligned}
& \mbox{dist}(x^{k+1}, \mathcal{X}) \leq \kappa\|\nabla f(x^{k+1})\| \\
& \qquad = \kappa\left\|\nabla f(x^{k+1}) - \nabla f(x^k) - \nabla^2 f(x^k)(x^{k+1} - x^k) - \frac{\sigma_k}{2}\|x^{k+1} - x^k\|(x^{k+1} - x^k)\right\| \\
& \qquad \leq \kappa \|\nabla f(x^{k+1}) - \nabla f(x^k) - \nabla^2 f(x^k)(x^{k+1} - x^k)\| + \frac{\kappa\sigma_k}{2}\|x^{k+1} - x^k\|^2 \\
& \qquad \leq \frac{\kappa(L + \sigma_k)}{2}\|x^{k+1} - x^k\|^2 \leq \frac{3}{2}\kappa L\|x^{k+1} - x^k\|^2,
\end{aligned}
\end{equation}
where the equality is due to the first-order optimality condition of \eqref{eq:true_cubic_model}, the third inequality is by~\eqref{eq:iterates-in-F} and Fact~\ref{fact:lip-result}, and the last inequality is by $\sigma_k\leq 2L$ for all $k\geq 0$. Using~\eqref{eq:iterates-in-F}, Lemma~\ref{lem:step-bd}, and $\sigma_k\geq \bar{\sigma}>0$ for all $k$, we get
\begin{equation}
\label{eq:step-bd-all}
\|x^{k+1} - x^k\| \leq c_1 \cdot \mbox{dist}(x^k,\mathcal{X}), \quad \forall k\geq \bar{k},
\end{equation}
where $c_1 =\left(1+\frac{L}{\bar{\sigma}} + \sqrt{\left(1 + \frac{L}{\bar{\sigma}}\right)^2 + \frac{L}{\bar{\sigma}}}\right)$. Combining~\eqref{eq:derive-quad-conv} and~\eqref{eq:step-bd-all}, we obtain 
\begin{equation}
\label{eq:dist-quad-conv}
\mbox{dist}(x^{k+1},\mathcal{X}) \leq c_2\cdot \mbox{dist}^2(x^k,\mathcal{X}), \quad \forall k\geq \bar{k},
\end{equation}
where $c_2 = \frac{3}{2}\kappa c_1^2L$. We next show that the whole sequence $\{x^k\}_{k\ge0}$ is convergent. Let $\eta>0$ be arbitrary. Since $\lim_{k\rightarrow \infty}\mbox{dist}(x^k,\mathcal{X}) = 0$, there exists a $k_3\geq 0$ such that 
$$\mbox{dist}(x^k,\mathcal{X})\leq \min\left\{ \frac{1}{2c_2}, \frac{\eta}{2c_1}\right\}, \quad \forall k\geq k_3. $$
It then follows from~\eqref{eq:dist-quad-conv} that
$$ \mbox{dist}(x^{k+1},\mathcal{X}) \leq c_2\cdot\mbox{dist}^2(x^k,\mathcal{X}) \leq \frac{1}{2}\mbox{dist}(x^k,\mathcal{X}), \quad \forall k\geq \max\{k_3,\bar{k}\}. $$
This, together with~\eqref{eq:step-bd-all}, implies that for any $k\geq \max\{k_3,\bar{k}\}$ and any $j\geq 0$, we have
\begin{equation}
\label{eq:derive-seq-conv}
\begin{aligned}
\|x^{k+j} - x^k\| & \leq \sum_{i=k}^{\infty}\|x^{i+1} - x^i\| \leq \sum_{i=k}^{\infty}c_1\cdot\mbox{dist}(x^i,\mathcal{X}) \\
& \leq c_1\cdot\mbox{dist}(x^k,\mathcal{X}) \cdot \sum_{i=0}^\infty \frac{1}{2^i} \leq 2c_1\cdot\mbox{dist}(x^k,\mathcal{X}) \leq \eta,
\end{aligned}
\end{equation} 
which implies that $\{x^k\}_{k\ge\max\{k_3,\bar{k}\}}$ is a Cauchy sequence. Therefore, the whole sequence $\{x^k\}_{k\ge0}$ is convergent. Finally, we study the convergence rate of $\{x^k\}_{k\ge0}$. Let $x^*:=\lim_{k\rightarrow\infty}x^k$. By Fact \ref{fact:CR_GC}, we have $x^*\in\mathcal{X}$. It follows from \eqref{eq:dist-quad-conv} and \eqref{eq:derive-seq-conv} that for any $k\geq \max\{k_3,\bar{k}\}$, 
$$ \|x^* - x^{k+1}\| = \lim_{j\rightarrow\infty}\|x^{k+1+j} - x^{k+1}\| \leq 2c_1\cdot\mbox{dist}(x^{k+1},\mathcal{X}) \leq 2c_1c_2\cdot\mbox{dist}^2(x^k,\mathcal{X}).$$
Combining this with $\mbox{dist}(x^k,\mathcal{X})\leq \|x^k - x^*\|$, we obtain
$$ \frac{\|x^{k+1} - x^*\|}{\|x^k - x^*\|^2} \leq 2c_1c_2, \quad \forall k\geq \max\{k_3,\bar{k}\}. $$
Therefore, $\{x^k\}_{k\ge0}$ converges at least Q-quadratically to an element $x^*$ in $\mathcal{X}$. 
\end{proof}

\smallskip
Equipped with Theorem \ref{thm:CR-LC-exact} and Propositions \ref{prop:eb-star} and \ref{prop:eb-grad-domi}, we can improve the results in \cite{NP06} on the local convergence rate of the CR method when applied to globally non-degenerate star-convex functions and gradient-dominated functions.
\begin{coro}
\label{cor:quad-conv-star-cvx}
Suppose that Assumption \ref{ass:hess-lip} holds, $f$ is star-convex, and $\mathcal{X}^*$ is globally non-degenerate. If $\mathcal{L}(f(x^k))$ is bounded for some $k\geq 0$, then the whole sequence $\{x^k\}_{k\ge0}$ converges at least Q-quadratically to a point $x^*\in\mathcal{X}^*$.
\end{coro}
\begin{proof}
Since $\mathcal{L}(f(x^k))$ is bounded for some $k\geq 0$, we have that $\mathcal{X}^*$ is bounded. This, together with Assumption \ref{ass:hess-lip}, implies that $\nabla^2 f(x)$ is uniformly continuous on $\mathcal{N}(\mathcal{X}^*;\gamma)$ for some $\gamma>0$. The premise of Proposition \ref{prop:eb-star} then holds. Hence, by Proposition \ref{prop:eb-star} and its proof, we obtain that Assumption \ref{ass:eb} holds and $\mathcal{X} = \mathcal{X}^*$. The conclusion of Corollary \ref{cor:quad-conv-star-cvx} then follows from Theorem \ref{thm:CR-LC-exact}.
\end{proof}

\begin{coro}
\label{cor:quad-conv-grad-domi}
Suppose that Assumption \ref{ass:hess-lip} holds and $f$ is gradient-dominated with degree $2$. If $\mathcal{L}(f(x^k))$ is bounded for some $k\geq 0$, then the whole sequence $\{x^k\}_{k\ge0}$ converges at least Q-quadratically to a point $x^*\in\mathcal{X}^*$. 
\end{coro}
\begin{proof}
Using the same arguments as those in the proof of Corollary \ref{cor:quad-conv-star-cvx}, we have that the premise of Proposition \ref{prop:eb-grad-domi} holds. Hence, by Proposition \ref{prop:eb-grad-domi} and its proof, we obtain that Assumption \ref{ass:eb} holds and $\mathcal{X} = \mathcal{X}^*$. The conclusion of Corollary \ref{cor:quad-conv-grad-domi} then follows from Theorem \ref{thm:CR-LC-exact}.
\end{proof}

\section{Applications to Structured Nonconvex Optimization Problems}\label{sec:app}
In this section, we study the CR method for solving two concrete nonconvex minimization problems that arise in phase retrieval and low-rank matrix recovery, respectively. 
\subsection{Phase Retrieval}\label{sec:PR}
In this subsection, we consider the application of the CR method for solving (noiseless) phase retrieval problems. Specifically, the problem of interest is to recover an unknown complex signal ${z^\star}={x^\star}+i{y^\star}\in\mathbb{C}^{n}$ from the measurements
\begin{equation}\label{eq:PR}
b_j = |a_j^H{z^\star}|,\quad j=1,\dots,m,
\end{equation}
where $\{a_j\}_{j=1}^m\subset \CC^n$ are assumed to be independently sampled from the standard complex Gaussian distribution $\CC\mathcal{N}(0,I_n)$. For non-triviality, we assume that $z^\star \neq 0$. Since for any $\phi\in[0,2\pi)$, $z^\star e^{i\phi}$ provides exactly the same measurements, we can only expect to recover $z^\star$ up to this ambiguity. Such problem has broad applications in science and engineering, including optics, signal processing, computer vision, and quantum mechanics. For more discussions on its applications and recent developments, the readers are invited to the survey papers \cite{SECCMS15,L17}.

Given the form~\eqref{eq:PR}, the following optimization formulation arises naturally:
\begin{equation}\label{opt:PR_complex}
\min_{z\in\mathbb{C}^{n}} f_c(z):= \frac{1}{2m}\sum_{j=1}^m \left( |a_j^Hz|^2 -b_j^2 \right)^2.
\end{equation}
Let $\mathcal{Z}^\star := \{ z^\star e^{i\phi} : \phi \in [0,2\pi)\}$ be the set of target signals. Observe that $f_c(z) \geq 0$ for any $z\in\CC^n$ and $f(z) = 0$ for any $z\in\mathcal{Z}^\star$. Hence, any $z\in \mathcal{Z}^\star$ is a global minimizer of~\eqref{opt:PR_complex}.
By letting $f(x,y) := f_c(x + iy)$, the corresponding real-variable problem of~\eqref{opt:PR_complex} is given by
\begin{equation}\label{opt:PR_real}
\min_{x,y\in\mathbb{R}^{n}} f(x,y) = \frac{1}{2m}\sum_{j=1}^m \left( \left\| 
\begin{pmatrix}
\Re(a_j) & -\Im(a_j) \\
\Im(a_j) & \Re(a_j)
\end{pmatrix}^T \begin{pmatrix}
x\\ y
\end{pmatrix} \right\|^2 -b_j^2 \right)^2.
\end{equation}
Let $\mathcal{X}^\star := \{ (x^\star \cos \phi - y^\star \sin \phi, \; x^\star \sin \phi + y^\star \cos \phi) : \phi \in [0,2\pi) \}$. Using similar arguments, one can verify that any $(x,y)\in\mathcal{X}^\star$ is a global minimizer of~\eqref{opt:PR_real}. Also, it holds that $(x,y)\in\mathcal{X}^\star$ if and only if $x + iy \in \mathcal{Z}^\star$.  Hence, as long as we obtain an element $(x,y)\in \mathcal{X}^\star$, the phase retrieval problem is solved by letting $z = x + iy$. %Note that since $f$ is a quatic polynomial, the CR method can be suitably applied for solving problem~\eqref{opt:PR_real}.

We now state the main result of this subsection. 
%Although problem~\eqref{opt:PR_real} is nonconvex, with the help of recent advances on the geometric analysis of the complex-variable function $f_c$~\cite{SQW16}, we shall show that when $m$ is sufficiently large,  with high probability, $\mathcal{X}^\star$ equals the set of all second-order critical points of $f$.  In addition, we shall prove that whenever such event happens, the error bound condition~\eqref{ineq:CR_EB} holds on a neighbourhood of $\mathcal{X}^\star$. These, together with our developments in Section~\ref{sec:CR_rate}, lead to the following result.
\begin{thm}
\label{thm:CR-PR}
There exist constants $c_0,c_1>0$ such that when $m\geq c_0 n \log^3 n$, it holds with probability at least $1 - c_1m^{-1}$ that with any arbitrary initialization, the sequence of iterates $\{(x^k, y^k)\}_{k\ge0}$ generated by Algorithm~\ref{alg:CR} for solving~\eqref{opt:PR_real} converges at least Q-quadratically to an element in $\mathcal{X}^\star$. 
\end{thm}

The rest of this subsection is devoted to proving the above theorem. Before we proceed, let us lay out the concepts of Wirtinger calculus that are necessary for our developments on complex-variable functions. Let $h_c:\mathbb{C}^{n} \rightarrow\mathbb{R}$ be a real-valued function on $\mathbb{C}^{n}$ and $h:\mathbb{R}^{2n}\rightarrow\mathbb{R}$ be defined as $h(x,y) = h_c(x+iy)$ for any $x,y\in\mathbb{R}^{n}$. We define 
%\begin{equation*}
%J := \frac{1}{2}\begin{pmatrix}
%I_{n} & iI_{n}\\
%I_{n} & -iI_{n}
%\end{pmatrix},
%\end{equation*}
\begin{equation*}
\frac{\partial}{\partial z} := \frac{1}{2}\left(\frac{\partial }{\partial x} - i\frac{\partial}{\partial y}\right)\text{ and }\frac{\partial}{\partial \bar{z}} := \frac{1}{2}\left(\frac{\partial }{\partial x} + i\frac{\partial}{\partial y}\right),
\end{equation*}
which can be understood as operators acting on real-valued functions of $(x,y)$. Then, the \emph{Wirtinger gradient} $\nabla_w h_c$ and \emph{Wirtinger Hessian} $\nabla_w^2 h_c$ of $h_c$ are defined, respectively, as 
\begin{equation*}
\nabla_w h_c := \left[ \frac{\partial h}{\partial z}, \frac{\partial h}{\partial \bar{z}} \right]^H \quad\text{ and }\quad \nabla_w^2 h_c =: \begin{pmatrix}
\frac{\partial}{\partial z}\left(\frac{\partial h}{\partial z}\right)^H & \frac{\partial}{\partial \bar{z}}\left(\frac{\partial h}{\partial z}\right)^H \\
\frac{\partial}{\partial z}\left(\frac{\partial h}{\partial \bar{z}}\right)^H & \frac{\partial}{\partial \bar{z}}\left(\frac{\partial h}{\partial \bar{z}}\right)^H
\end{pmatrix}.
\end{equation*}
Define the matrix 
\begin{equation*}
J = \frac{1}{2}
\begin{pmatrix}
I_n & i I_n \\
I_n & -i I_n
\end{pmatrix},
\end{equation*}
which satisfies $2JJ^H = 2J^HJ = I_{2n}$
%and for any $z = x + iy \in \CC^n$, it holds that
%\begin{equation}
%\label{eq:property-J}
%2J\begin{pmatrix}
%x \\ y
%\end{pmatrix} = \begin{pmatrix}
%z \\ \overline{z}
%\end{pmatrix}, \quad
%J^H \begin{pmatrix}
%z \\ \overline{z}
%\end{pmatrix}
% = \begin{pmatrix}
% x \\ y
% \end{pmatrix}.
%\end{equation} 
%In addition, 
and
$$%\begin{equation} \label{eq:identity-wt}
\nabla_w h_c(x + iy) = J\nabla h(x,y), \quad \nabla_w^2 h_c(x + iy) = J\nabla^2 h(x,y) J^H.
$$%\end{equation}
In particular, for the function $f_c$ in~\eqref{opt:PR_complex}, we have
\begin{equation}
\label{eq:wt-grad-PR}
\nabla_w f_c(z) = \frac{1}{m}\sum_{j=1}^m \begin{pmatrix}
\left( |a_j^Hz|^2 - b_j^2 \right)(a_ja_j^H)z\\
\left( |a_j^Hz|^2 - b_j^2 \right)(a_ja_j^H)^T\overline{z}
\end{pmatrix}
\end{equation}
and
\begin{equation}
\label{eq:wt-hess-PR}
\nabla_w^2 f_c(z) = \frac{1}{m} \sum_{j=1}^m \begin{pmatrix}
\left( 2|a_j^Hz|^2 - b_j^2 \right)a_ja_j^H & (a_j^Hz)^2 a_ja_j^T\\
(a_j^T\overline{z})^2 \overline{a_j}a_j^H & \left( 2 |a_j^Hz|^2 - b_j^2 \right)\overline{a_j}a_j^T\end{pmatrix}
\end{equation}
for any $z\in\CC^n$; see, \eg, \cite[Section 7.2]{CLS15}.

\subsubsection{Second-Order Critical Points and Local EB Condition}
We first show that with high probability, the set of second-order critical points of $f$ is $\mathcal{X}^\star$. Moreover, we show that in a neighbourhood of $\mathcal{X}^\star$, the local EB condition~\eqref{eq:eb} holds. For this purpose, we need the following result, which is directly implied by~\cite[Theorem 2]{SQW16}.
\begin{fact}
\label{fact:SQW}
Let $\mathcal{U}_c$ be a neighbourhood of $\mathcal{Z}^\star$ defined as
$\mathcal{U}_c := \big\{z \in \CC^n : \dist\left(z, \mathcal{Z}^\star\right) \leq \frac{1}{\sqrt{7}}\|z^\star\| \big\}. $ There exist constants $c_2,c_3>0$ such that when $m\geq c_2 n \log^3 n$, the following statements hold with probability at least $1 - c_3m^{-1}$.
\begin{enumerate}
\item[(i)] For any $z\notin \mathcal{U}_c$, if $\nabla_w f_c(z) = 0$, then
\begin{equation}
\label{eq:neg-curve}
\begin{pmatrix}
\hat{z} \smallskip\\ \overline{\hat{z}}
\end{pmatrix}^H 
\nabla_w^2 f_c(z) 
\begin{pmatrix}
\hat{z} \smallskip\\ \overline{\hat{z}}
\end{pmatrix} 
\leq -\frac{\|z^\star\|^4}{100},
\end{equation}
where $\hat{z}$ is defined as the unique projection of $z$ to $\mathcal{Z}^\star$; \ie,
$$ \hat{z} = z^\star e^{i\phi(z)}, \;\;\mbox{with}\,\;\phi(z) = \argmin_{\phi\in[0,2\pi)} \left\| z - z^\star e^{i\phi} \right\|. $$
\item[(ii)] For any $z\in \mathcal{U}_c$, it holds that
\begin{equation}
\label{eq:pos-curve}
\begin{pmatrix}
g(z) \smallskip\\ \overline{g(z)}
\end{pmatrix}^H 
\nabla_w^2 f_c(z) 
\begin{pmatrix}
g(z) \smallskip\\ \overline{g(z)}
\end{pmatrix} 
\geq \frac{\|z^\star\|^2}{4}\cdot\|g(z)\|^2,
\end{equation}
where $g(z):= z - \hat{z}$. 
%is defined as
%$$
%g(z) := \left\{ 
%\begin{array}{ll}
%z - \hat{z} & \mbox{if}\; z \notin \mathcal{Z}^\star, \\
%g \in \mathcal{S} := \{ g: \Im(g^Hz)=0 \} & \mbox{if}\; z \in \mathcal{Z}^\star.
%\end{array}
%\right.
%$$
\end{enumerate}
\end{fact}

\begin{prop}
\label{prop:geometry-f-PR}
There exist constants $c_2,c_3>0$ such that when $m\geq c_2 n \log^3 n$, the following statements on $f$ hold with probability at least $1 - c_3m^{-1}$.
\begin{enumerate}
\item[(i)] $\mathcal{X}^\star$ equals the set of second-order critical points of $f$.
\item[(ii)] The following error bound holds:
\begin{equation}
\label{eq:eb-PR}
\dist\big( (x,y), \mathcal{X}^\star \big) \leq \frac{4}{\|z^\star\|^2} \|\nabla f(x,y) \| \quad \mbox{whenever}\;\; \dist\big((x,y), \mathcal{X}^\star\big) \leq \frac{1}{\sqrt{7}}\|z^\star\|.
\end{equation}
\end{enumerate}
\end{prop}
\begin{proof}
It suffices to prove that statements (i) and (ii) of Fact~\ref{fact:SQW} lead to the statements (i) and (ii) herein. We first prove (ii). Let $(x,y)$ be an arbitrary point satisfying $\mbox{dist}\big((x,y), \mathcal{X}^\star\big) \leq \frac{1}{\sqrt{7}}\|z^\star\|$ and $(\hat{x},\hat{y})$ be the projection of $(x,y)$ to $\mathcal{X}^\star$. By definition, one can easily verify that $x + iy \in \mathcal{U}_c$ and the projection of $x + iy$ to $\mathcal{Z}^\star$ is $\hat{x} + i\hat{y}$. We assume that $(x,y) \notin \mathcal{X}^\star$ since~\eqref{eq:eb-PR} holds trivially otherwise. Thus, $x + iy \notin \mathcal{Z}^\star$ and we have $g(x + iy) = x - \hat{x} + i(y - \hat{y})$. This, together with the identity $\nabla_w^2 f_c(z) = J\nabla^2 f(x,y) J^H$ and~\eqref{eq:pos-curve}, yields
\begin{equation}
\label{eq:posi-curv-real}
\begin{pmatrix}
x - \hat{x} \\ y - \hat{y}
\end{pmatrix}^T \nabla^2 f(x,y) 
\begin{pmatrix}
x - \hat{x} \\ y - \hat{y}
\end{pmatrix} \geq \frac{\|z^\star\|^2}{4} \cdot \left\|(x,y) - (\hat{x}, \hat{y})\right\|^2.
\end{equation}
Let $(x(t), y(t)) = t\cdot (x,y) + (1 - t) \cdot (\hat{x}, \hat{y})$ for $t\in[0,1]$. Note that the projection of $(x(t),y(t))$ to $\mathcal{X}^\star$ is $(\hat{x},\hat{y})$ for any $t\in[0,1]$. Using the same arguments, \eqref{eq:posi-curv-real} holds if we substitute $(x,y)$ by $(x(t), y(t))$ for any $t\in[0,1]$. Hence, by the integral form of Taylor's series, we obtain
\begin{equation*}
\begin{aligned}
f(x,y) & = f(\hat{x},\hat{y}) + \nabla f(\hat{x}, \hat{y})^T \begin{pmatrix}
x - \hat{x} \\ y - \hat{y}
\end{pmatrix} + \int_0^1 (1-t)\cdot\begin{pmatrix}
x - \hat{x} \\ y - \hat{y}
\end{pmatrix}^T\nabla^2 f(x(t),y(t)) \begin{pmatrix}
x - \hat{x} \\ y - \hat{y}
\end{pmatrix} dt \\
& \geq f(\hat{x},\hat{y}) + \nabla f(\hat{x}, \hat{y})^T \begin{pmatrix}
x - \hat{x} \\ y - \hat{y}
\end{pmatrix} + \frac{\|z^\star\|^2}{8}\cdot \left\|(x,y) - (\hat{x}, \hat{y})\right\|^2,
\end{aligned}
\end{equation*} 
and similarly, 
\begin{equation*}
\begin{aligned}
f(\hat{x},\hat{y}) \geq f(x,y) - \nabla f(x,y)^T \begin{pmatrix}
x - \hat{x} \\ y - \hat{y}
\end{pmatrix} + \frac{\|z^\star\|^2}{8}\cdot \left\|(x,y) - (\hat{x}, \hat{y})\right\|^2.
\end{aligned}
\end{equation*} 
Noticing that $f(\hat{x},\hat{y}) = 0$ and $\nabla f(\hat{x}, \hat{y}) = 0$ (by the global optimality of $(\hat{x},\hat{y})$), we obtain~\eqref{eq:eb-PR}  by summing up the above two inequalities. 

We next prove (i). Let $\mathcal{X}$ be the set of second-order critical points of $f$. Clearly, we have $\mathcal{X}^\star\subset \mathcal{X}$ since any $(x,y)\in\mathcal{X}^\star$ is a global minimizer of $f$. We now show that $\mathcal{X} \subset \mathcal{X}^\star$. Let $(x,y)\in\mathcal{X}$ be arbitrary. By definition, $\nabla f(x,y) = 0$ and $\nabla^2 f(x,y) \succeq 0$. Using $\nabla f(x,y) = 0$ and the result in (i), we see that $(x,y) \in\mathcal{X}^\star$ or $\mbox{dist}\big((x,y), \mathcal{X}^\star\big) > \frac{1}{\sqrt{7}}\|z^\star\|$. If the latter holds, we have by definition that $x + iy \notin \mathcal{U}_c$.  In addition, it holds that $\nabla_w f_c(x + iy) = J\nabla f(x,y) = 0$. Hence, the inequality~\eqref{eq:neg-curve} holds for $x + iy$. This, together with the identity $\nabla_w^2 f_c(x + iy) = J\nabla^2 f(x,y) J^H$, implies that $\nabla^2 f(x,y) \nsucceq 0$, which contradicts with $(x,y)\in\mathcal{X}$. Therefore, we have $\mathcal{X}\subset\mathcal{X}^\star$ and hence $\mathcal{X} = \mathcal{X}^\star$.
\end{proof}

\subsubsection{Lipschitz Continuity of $\nabla^2 f$}
Our next step is to verify the Lipschitz continuity of $\nabla^2 f$. Let $A = (a_1, \ldots, a_m) \in \CC^{n\times m}$ and $M = \max_j \|a_j\|$. We need the following result, which combines Lemma 23 and Lemma 28 of~\cite{SQW16}.
\begin{fact}\label{fact:CR_3}
There exist constants $c_4,c_5,c_6>0$ such that when $m\ge c_4 n $, it holds with probability at least $1-c_5\exp(-c_6m)$ that
\begin{equation}
\label{eq:eigenvalue-bd}
\frac{m}{2} \leq \lambda_{\min}(AA^H) \leq \lambda_{\max}(AA^H) \leq 2m
\end{equation}
and
\begin{equation}
\label{eq:con-for-smooth}
\frac{1}{m}\sum_{j=1}^m\left| |a_j^Hw|^2 - |a_j^Hw^\prime|^2 \right| \le \frac{3}{2}\|w - w^\prime\| \left(\|w\| + \|w^\prime\|\right), \quad \forall w, w^\prime \in \CC^n.
\end{equation}
\end{fact}
\begin{prop}\label{prop:Lip}
Suppose that~\eqref{eq:eigenvalue-bd} and~\eqref{eq:con-for-smooth} hold. Then, for any $R>0$, $\nabla^2 f$ is Lipschitz continuous on $\mathbb{B}(0;R) = \left\{(x,y)\in\mathbb{R}^{2n}: \|(x,y)\|\leq R \right\}$ with Lipschitz constant $L = 20M^2R$.
\end{prop}

\begin{proof} 
Let $w = x+ iy$ and $w^\prime = x^\prime + i y^\prime$ with $(x,y),(x^\prime,y^\prime)\in \mathbb{B}(0;R)$. Thus, $\|w\| \leq R$ and $\|w^\prime\|\leq R$. 
By the identities $\nabla_w^2 f_c (x + iy) = J\nabla^2 f(x,y) J^H$ and $2J^HJ = I_{2n}$, we have
\begin{equation*}
\begin{aligned}
\| \nabla^2 f(x,y) - \nabla^2 f(x^\prime, y^\prime) \| & = \sup_{\|(u,v)\| = 1} \left| \begin{pmatrix}
u \\ v
\end{pmatrix}^T \left[ \nabla^2 f(x,y) - \nabla^2 f(x^\prime, y^\prime) \right] \begin{pmatrix}
u \\ v
\end{pmatrix} 
\right| \\
& = \sup_{\|(u,v)\| = 1} \left| 4\begin{pmatrix}
u \\ v
\end{pmatrix}^T J^H \left[ \nabla_w^2 f_c(w) - \nabla_w^2 f_c(w^\prime) \right] J \begin{pmatrix}
u \\ v
\end{pmatrix} 
\right| \\
& = \sup_{\|z\| = 1} \left| \begin{pmatrix}
z \\ \overline{z}
\end{pmatrix}^H \left[ \nabla_w^2 f_c(w) - \nabla_w^2 f_c(w^\prime) \right] \begin{pmatrix}
z \\ \overline{z}
\end{pmatrix} 
\right|.
\end{aligned}
\end{equation*}
Using~\eqref{eq:wt-hess-PR},~\eqref{eq:eigenvalue-bd}, and~\eqref{eq:con-for-smooth}, we further have
\begin{equation*}
\begin{aligned}
& \| \nabla^2 f(x,y) - \nabla^2 f(x^\prime, y^\prime) \|  \\
&  \qquad \leq  \sup_{\|z\| = 1 }\Big| \frac{4}{m}\sum_{j=1}^m \left(|a_j^Hw|^2 - |a_j^Hw^\prime|^2\right) |a_j^Hz|^2 \Big| + \Big| \frac{2}{m} \sum_{j=1}^m \Re \Big( \left[ (a_j^Hw)^2 - (a_j^Hw^\prime)^2 \right] (z^H a_j)^2 \Big) \Big|\\
& \qquad \leq  \sup_{\|z\| = 1 } \frac{4}{m}\sum_{j=1}^m \left| |a_j^Hw|^2 - |a_j^Hw^\prime|^2 \right| |a_j^Hz|^2 + \frac{2}{m}\sum_{j=1}^m | a_j^H w - a_j^H w^\prime | | a_j^Hw + a_j^H w^\prime | |a_j^H z|^2 \\
& \qquad \leq 4M^2\cdot\frac{1}{m}\sum_{j=1}^m \left| |a_j^Hw|^2 - |a_j^Hw^\prime|^2 \right| + 4M^2R\|w - w^\prime\|\cdot\Big\|\frac{1}{m} \sum_{j=1}^m a_ja_j^H \Big\| \\
& \qquad \leq 6M^2\|w - w^\prime\|(\|w\| + \|w^\prime\|) + \frac{4}{m}M^2 R \|w - w^\prime\|\cdot \lambda_{\max}(AA^H) \\
%& \qquad \leq 12M^2R\|(x,y)-(x',y')\| + 8M^2R\|(x,y)-(x',y')\| \\
& \qquad \le 20 M^2R\|(x,y)-(x',y')\|.
\end{aligned}
\end{equation*}
%holds with probability at least $1-c_5\exp(-c_6m)$, provided that $m\geq c_4n$.
The proof is then completed.
\end{proof}

\subsubsection{Proof of Theorem~\ref{thm:CR-PR}}
In view of Theorem~\ref{thm:CR-LC-exact}, it suffices to prove that  with high probability, the following statements hold simultaneously: (i) Assumption~\ref{ass:hess-lip} holds, (ii) Assumption~\ref{ass:eb} holds, (iii) $\mathcal{L}(f(x^k,y^k))$ is bounded for some $k\geq 0$, and (iv) $\mathcal{X}^\star$ equals the set of second-order critical points of $f$.

\smallskip
Let $c_0 := \max\{c_2,c_4\}$ and $m\geq c_0n\log^3 n$. Suppose that $(x^0,y^0)\in\mathbb{R}^n\times\mathbb{R}^n$ is the initial point of Algorithm~\ref{alg:CR}. 
%Let $\bar{R}>0$ be the square root of $\frac{2\sqrt{2}}{\sqrt{m}} \sqrt{2m f(x^0,y^0) +\sum_{j=1}^m b_j^4}$ and
We define
\begin{align*}
\bar{R} & := \left( \frac{2\sqrt{2}}{\sqrt{m}} \sqrt{2m f(x^0,y^0) +\sum_{j=1}^m b_j^4} \ \right)^{\frac{1}{2}} > 0, \\
\mathcal{F} & := \mathbb{B}(0;2\bar{R}) =\{(x,y)\in\mathbb{R}^n\times\mathbb{R}^n : \|(x,y)\| \leq 2\bar{R} \}.
\end{align*}
Suppose that~\eqref{eq:eigenvalue-bd} and~\eqref{eq:con-for-smooth} hold. Let $(x,y)$ be an arbitrary point in $\mathcal{L}(f(x^0,y^0))$ and set $z := x + iy$. By definition, we have $f_c(z) = f(x, y) \leq f(x^0,y^0)$. It then follows from~\eqref{eq:eigenvalue-bd} and the definition of $A$ that
$$ 
\begin{aligned}
\|(x,y)\|^2 = \|z\|^2 & \leq \frac{1}{\lambda_{\min}(AA^H)}\cdot z^HAA^Hz \leq \frac{2}{m} \sum_{j=1}^m |a_j^H z|^2.
\end{aligned}
$$
Using the inequality $(\sum_{j=1}^m \alpha_j)^2 \leq m\sum_{j=1}^m \alpha_j^2$, which holds for any real numbers $\{\alpha_i\}_{i=1}^m$, we further have
$$ 
\begin{aligned}
\|(x,y)\|^2 & \leq \frac{2}{\sqrt{m}}\cdot\sqrt{\sum_{j=1}^m |a_j^Hz|^4} \leq \frac{2}{\sqrt{m}}\cdot\sqrt{\sum_{j=1}^m \left[2\left(|a_j^Hz|^2 - b_j^2\right)^2 + 2b_j^4\right]} \\
& = \frac{2\sqrt{2}}{\sqrt{m}} \sqrt{2m f_c(z) +\sum_{j=1}^m b_j^4} \leq \frac{2\sqrt{2}}{\sqrt{m}} \sqrt{2m f(x^0,y^0) +\sum_{j=1}^m b_j^4} = \bar{R}^2.
\end{aligned}
$$
By the definition of $\mathcal{F}$, we have $\mathcal{L}(f(x^0,y^0))\subset\mbox{int}(\mathcal{F})$ and $\mathcal{L}(f(x^0,y^0))$ is bounded. In addition, by Proposition~\ref{prop:Lip}, $\nabla^2 f$ is Lipschitz continuous on $\mathcal{F}$ with Lipschitz constant $L = 40M^2\bar{R}$. Hence, statements (i) and (iii) above hold with probability at least $1 - c_5\exp(-c_6m)$. Furthermore, by Proposition~\ref{prop:geometry-f-PR}, statements (ii) and (iv) above hold with probability at least $1 - c_4m^{-1}$. Therefore, there exists a $c_1>0$ such that all the statements (i)--(iv) hold with probability at least $1 - c_1m^{-1}$. The proof is then completed. 
\endproof

\subsection{Low-Rank Matrix Recovery}\label{sec:LRM}
In this subsection, we consider the application of the CR method for solving low-rank matrix recovery problems. Specifically, the problem of interest is to recover an unknown low-rank matrix $X^\star \in \mathbb{R}^{n_1\times n_2}$ with $\mbox{rank}(X^\star) = r\ll \min\{n_1,n_2\}$ from the measurements
\begin{equation}
\label{eq:lr-mat-measure}
\mathbb{R}^m \ni b = \mathcal{A}(X^\star),
\end{equation}
where the linear operator $\mathcal{A}:\mathbb{R}^{n_1\times n_2} \rightarrow \mathbb{R}^m$ is given by $\mathcal{A}(X) = \left( \langle A_1, X \rangle, \ldots, \langle A_m, X \rangle \right)$ for any $X\in\mathbb{R}^{n_1 \times n_2}$. 
%Such problem has found numerous applications in various fields such as system identification~\cite{liu2009interior}, multi-task learning~\cite{argyriou2008convex}, recommendation systems~\cite{koren2009matrix}, and sensor network localization~\cite{biswas2004semidefinite,so2007theory}. 
For simplicity, we assume that $n_1 = n_2 = n$, $A_i$'s are symmetric, and the target matrix $X^\star$ is symmetric and positive semidefinite.
%We remark that the results developed in this section can be extended to rectangular matrices.
%We note that the results developed in this section can also be extended to the case where $X^\star$ is a rectangular matrix with a modified objective and more involved analysis  (see, {\it e.g.},~\cite{zhu2017global}). 

%Recently, there is an increasing interest in nonconvex optimization formulations for solving the low-rank matrix recovery problem, due to their low per-iteration cost and remarkable performance in practice (see, {\it e.g.},~\cite{chen2015fast,sun2016guaranteed}). 
Since $X^\star\succeq 0$ with $\mbox{rank}(X^\star) = r$, we have $X^\star = U^\star {U^\star}^T$ for some $U^\star\in\mathbb{R}^{n\times r}$. This motivates the following nonconvex formulation for recovering $X^\star$:
\begin{equation}
\label{eq:ncvx-lr-mat}
\min_{U\in\mathbb{R}^{n\times r}} f(U) := \frac{1}{4m}\left\| \mathcal{A}(UU^T) - b \right\|^2.
\end{equation}
By letting $\mathcal{U} := \{ U^\star Q: Q \in \mathcal{O}^r \}$, it holds that $\mathcal{U} = \{ U\in\mathbb{R}^{n\times r}: X^\star = UU^T \}$. Hence, we can recover the unknown matrix $X^\star$ as long as we find any $U\in\mathcal{U}$. Observe that $f$ is non-negative and $f(U) = 0$ for any $U\in\mathcal{U}$. Hence, any $U\in \mathcal{U}$ is a global minimizer of~\eqref{eq:ncvx-lr-mat}.

We next introduce the so-called restricted isometry property (RIP) of the operator $\mathcal{A}$.

\begin{defi} \label{defi:rip}
We say that the linear operator $\mathcal{A}$ satisfies $(r,\delta_r)$-RIP if for any matrix $X\in\mathbb{R}^{n\times n}$ with ${\rm rank}(X)\leq r$,
$$ (1 - \delta_r)\|X\|_F^2 \leq \frac{1}{m}\sum_{i=1}^m \langle A_i, X \rangle^2 \leq (1 + \delta_r)\|X\|_F^2. $$
\end{defi}
The above definition has played an important role in the literature of low-rank matrix recovery. One well-known case where the RIP holds is when $\mathcal{A}$ is a random measurement operator. For example, if $\{A_i\}_{i=1}^m$ are mutually independent random Gaussian matrices, then when $m\geq Dnr$, $\mathcal{A}$ satisfies the RIP for some $\delta_r<1$ with probability at least $1 - C\exp(-dm)$, where $C,D,d$ are absolute positive scalars~\cite[Theorem 2.3]{CP11}.

We now state the main result of this subsection. 
%Although~\eqref{eq:ncvx-lr-mat} is nonconvex, we shall prove that under the RIP of $\mathcal{A}$, $\mathcal{U}$ equals the set of second-order critical points of~\eqref{eq:ncvx-lr-mat}. By Theorem~\ref{thm:CR_GC}, this enables the CR method to solve~\eqref{eq:ncvx-lr-mat} globally. In addition, we shall prove that in a neighbourhood of $\mathcal{U}$, the error bound condition holds, which, together with Theorem~\ref{thm:CR_LC}, leads to the quadratic convergence of the CR method.

\begin{thm}
\label{thm:CR-LRM}
Suppose that $\mathcal{A}$ satisfies $(2r,\delta_{2r})$-RIP with $\delta_{2r}<\frac{1}{10}$. Then, with any arbitrary initialization, the sequence of iterates $\{U^k\}_{k\ge0}$ generated by Algorithm~\ref{alg:CR} for solving~\eqref{eq:ncvx-lr-mat} converges at least Q-quadratically to an element in $\mathcal{U}$.
\end{thm}

The rest of this subsection is devoted to proving Theorem~\ref{thm:CR-LRM}. Before we proceed, let us introduce some notations and preliminaries. Since $X^\star\succeq 0$ and $\mbox{rank}(X^\star) = r$, we have $\lambda_1(X^\star) \geq \cdots\geq \lambda_r(X^\star) >0$. As a result, the singular values of any $U\in\mathcal{U}$ are $\{\sqrt{\lambda_i(X^\star)}\}_{i=1}^r$. Let $\nabla f(U)\in\mathbb{R}^{n\times r}$ and $\nabla^2 f(U)\in\mathbb{R}^{(nr)\times(nr)}$ be the gradient and Hessian of $f$ at $U$, respectively. %By Taylor's theorem, 
%$$ f(U + Z) = f(U) + \langle \nabla f(U), Z\rangle + \mbox{vec}(Z)^T\nabla^2 f(U) \mbox{vec}(Z) + \mathcal{O}(\|Z\|_F^3), $$
For problem~\eqref{eq:ncvx-lr-mat}, a routine calculation gives
\begin{equation}
\label{eq:grad-lrm}
\nabla f(U) = \frac{1}{m}\sum_{i=1}^m \langle A_i, UU^T - X^\star\rangle A_iU
\end{equation}
and 
\begin{equation}
\label{eq:hess-lrm}
\begin{aligned}
\mbox{vec}(Z)^T\nabla^2 f(U) \mbox{vec}(Z) & = \left\langle Z,  \lim_{\tau\rightarrow 0} \frac{\nabla f(U + \tau Z) - \nabla f(U)}{\tau}\right\rangle \\
& = \frac{1}{m}\sum_{i=1}^m 2\langle A_i, UZ^T \rangle^2 + \langle A_i, UU^T - X^\star \rangle \langle A_i, ZZ^T \rangle,
\end{aligned}
\end{equation}
where $\mbox{vec}(Z)\in\mathbb{R}^{nr}$ is the vector obtained by stacking the columns of $Z$. The following result, which is stated in \cite[Lemma 4.1]{BNS16} and is related to~\cite[Lemma 2.1]{C08}, is crucial to our analysis.

\begin{fact}
\label{fact:rip}
For any $X,Y\in\mathbb{R}^{n\times n}$ with ${\rm rank}(X),{\rm rank}(Y) \le r$, if $\mathcal{A}$ is $(2r,\delta_{2r})$-RIP, then it holds that
$$%\begin{equation} \label{eq:rip-result}
\left| \frac{1}{m} \sum_{i=1}^m \langle A_i, X \rangle \langle A_i, Y \rangle - \langle X, Y \rangle \right| \leq \delta_{2r} \|X\|_F\|Y\|_F.
$$%\end{equation}
\end{fact}

\subsubsection{Second-Order Critical Points and Local EB Condition}
We first show that under the RIP, the set of second-order critical points of $f$ is $\mathcal{U}$. Moreover, we show that in a neighbourhood of $\mathcal{U}$, the local EB condition~\eqref{eq:eb} holds. The following result is due to~\cite[Theorem 3.2]{BNS16}.
\begin{fact}
\label{fact:stric-sad-lrm}
Suppose that $\mathcal{A}$ satisfies $(2r, \delta_{2r})$-RIP with $\delta_{2r} < \frac{1}{10}$. Then, for any $U\in\mathbb{R}^{n\times r}$ such that $\nabla f(U) = 0$ and $UU^T\neq X^\star$, it holds that
$$ \lambda_{\min} \left( \nabla^2 f(U) \right) \leq -\frac{\lambda_r(X^\star)}{5}<0. $$
%where $\lambda_{\min}\left(\nabla^2 f(U)\right)$ is the smallest eigenvalue of $\nabla^2 f(U)$.
\end{fact}

\begin{prop}
\label{prop:second-order=global-lrm}
Suppose that $\mathcal{A}$ satisfies $(2r, \delta_{2r})$-RIP with $\delta_{2r} < \frac{1}{10}$. Then, the following statements hold.
\begin{enumerate}
\item[(i)] $\mathcal{U}$ equals the set of second-order critical points of $f$.
\item[(ii)] The following error bound holds:
\begin{equation}
\label{eq:eb-lrm}
\dist(U, \mathcal{U}) \leq \frac{2}{\lambda_r(X^\star)}\|\nabla f(U)\| \quad \mbox{whenever}\;\; \dist(U,\mathcal{U})\leq \frac{1}{3}\sqrt{\lambda_r(X^\star)}.
\end{equation} 
\end{enumerate}
\end{prop}
\begin{proof}
By the global optimality of $\mathcal{U}$, any $U\in\mathcal{U}$ is a second-order critical point of $f$. On the other hand, due to Fact~\ref{fact:stric-sad-lrm}, any $U\notin\mathcal{U}$ cannot be a second-order critical point of $f$ if $\mathcal{A}$ satisfies $(2r,\delta_{2r})$-RIP with $\delta_{2r}<\frac{1}{10}$. Therefore, the result in (i) holds.

We next prove (ii). Let $\hat{U}$ be the projection of $U$ to $\mathcal{U}$ and $\Delta := U - \hat{U}$. Clearly, $\mbox{dist}(U,\mathcal{U}) = \|\Delta\|_F$. By the definition of $\mathcal{U}$, we have $\hat{U} = U^\star \bar{Q}$, where $\bar{Q} = \argmin_{Q\in\mathcal{O}^r} \| U - U^\star Q\|^2$.
Let ${U^\star}^TU = P\Sigma R^T$ be the singular value decomposition of ${U^\star}^TU$, where $\Sigma\in\mathbb{R}^{r\times r}$ is diagonal and $P,R\in\mathcal{O}^r$. Then, one can verify that $\bar{Q} = PR^T$. Hence, it follows that
$$ \Delta^T\hat{U} = (U - U^\star\bar{Q})^TU^\star\bar{Q} = R\Sigma R^T - \hat{U}^T\hat{U} = \hat{U}^T\Delta. $$
Using $\mathcal{A}(X^\star) = \mathcal{A}(\hat{U}\hat{U}^T) = b$ and \eqref{eq:grad-lrm}, one has $\nabla f(U) = \frac{1}{m}\sum_{i=1}^m \langle A_i, UU^T - \hat{U}\hat{U}^T \rangle A_i U$. It then follows from Fact~\ref{fact:rip} that
\begin{equation}
\label{eq:grad-lb-lrm}
\begin{aligned}
\langle \nabla f(U), \Delta\rangle & = \frac{1}{m} \sum_{i=1}^m \langle A_i, UU^T - \hat{U}\hat{U}^T \rangle\langle A_i, \Delta U^T \rangle \\
& \geq \langle UU^T - \hat{U}\hat{U}^T, \Delta U^T \rangle - \delta_{2r}\|UU^T - \hat{U}\hat{U}^T\|_F\|\Delta U^T\|_F.
\end{aligned}
\end{equation}
Using $U = \Delta + \hat{U}$ and $\Delta^T\hat{U} = \hat{U}^T\Delta$, we obtain
$$ 
\begin{aligned}
\langle UU^T - \hat{U}\hat{U}^T, \Delta U^T \rangle & = \langle \Delta \Delta^T + \hat{U}\Delta^T + \Delta\hat{U}^T, \Delta\Delta^T + \Delta\hat{U}^T \rangle \\
& = \langle \Delta \Delta^T, \Delta \Delta^T \rangle + 3 \langle \hat{U}\Delta^T, \Delta \Delta^T \rangle + 2\|\hat{U}\Delta^T\|_F^2 \\
& \geq \|\Delta\Delta^T\|_F^2 - 3\|\hat{U}\Delta^T\|_F\|\Delta\Delta^T\|_F + 2\|\hat{U}\Delta^T\|_F^2. 
\end{aligned}
$$
Also, we have
$$
\begin{aligned}
\|UU^T - \hat{U}\hat{U}^T\|_F\|\Delta U^T\|_F & = \|\Delta \Delta^T + \hat{U}\Delta^T + \Delta\hat{U}^T\|_F\|\Delta\Delta^T + \Delta\hat{U}^T\|_F \\
& \leq \|\Delta\Delta^T\|_F^2 + 3\|\hat{U}\Delta^T\|_F\|\Delta\Delta^T\|_F + 2\|\hat{U}\Delta^T\|_F^2. 
\end{aligned}
$$
Hence, it follows from from~\eqref{eq:grad-lb-lrm} that
\begin{equation}
\label{eq:grad-lb-lrm-2}
\begin{aligned}
\langle \nabla f(U), \Delta \rangle & \geq (1 - \delta_{2r})\|\Delta\Delta^T\|_F^2 - 3(1 + \delta_{2r})\|\hat{U}\Delta^T\|_F\|\Delta\Delta^T\|_F + 2(1 - \delta_{2r})\|\hat{U}\Delta^T\|_F^2  \\
& \geq \|\hat{U}\Delta^T\|_F \left( 2(1 - \delta_{2r})\|\hat{U}\Delta^T\|_F - 3(1 + \delta_{2r})\|\Delta\Delta^T\|_F\right),
\end{aligned}
\end{equation}
where the second inequality uses $\delta_{2r}<1$. Since the smallest singular value of $\hat{U}$ is $\sqrt{\lambda_r(X^\star)}$, it holds that $\|\hat{U}\Delta^T\|_F \geq \sqrt{\lambda_r(X^\star)}\|\Delta\|_F$. This, together with $\|\Delta\|_F\leq \frac{1}{3}\sqrt{\lambda_r(X^\star)}$ and $\delta_{2r}<\frac{1}{10}$, gives
$$ 
\begin{aligned}
2(1 - \delta_{2r})\|\hat{U}\Delta^T\|_F - 3(1 + \delta_{2r})\|\Delta\Delta^T\|_F & \geq 2(1 - \delta_{2r})\sqrt{\lambda_r(X^\star)}\|\Delta\|_F - (1+\delta_{2r})\sqrt{\lambda_r(X^\star)}\|\Delta\|_F \\
& \geq \frac{\sqrt{\lambda_r(X^\star)}}{2}\|\Delta\|_F.
\end{aligned}
$$
Substituting this into~\eqref{eq:grad-lb-lrm-2} and using $\|\hat{U}\Delta^T\|_F \geq \sqrt{\lambda_r(X^\star)}\|\Delta\|_F$, we obtain 
$$\langle \nabla f(U), \Delta \rangle \geq \frac{\lambda_r(X^\star)}{2}\|\Delta\|_F^2, $$
which, together with the Cauchy-Schwarz inequality, implies the required error bound~\eqref{eq:eb-lrm}.
\end{proof}

\subsubsection{Lipschitz Continuity of $\nabla^2 f$}
We next verify the Lipschitz continuity of $\nabla^2 f$.
\begin{prop}
\label{prop:smooth-lrm}
Suppose that $\mathcal{A}$ satisfies $(2r,\delta_{2r})$-RIP with $\delta_{2r}<\frac{1}{10}$. Then, for any $R>0$, $\nabla^2 f$ is Lipschitz continuous on $\mathbb{B}(0;R)= \left\{ U\in\mathbb{R}^{n\times r}: \|U\|_F\leq R\right\}$ with Lipschitz constant $L = 5R$.
\end{prop}
\begin{proof}
Let $U,U^\prime\in\mathbb{B}(0;R)$. Hence, $\|U\|_F\leq R$ and $\|U^\prime\|_F \leq R$. 
%It then follows from~\eqref{eq:grad-lrm} that
%\begin{equation*}
%\begin{aligned}
%& \|\nabla f(U) - \nabla f(U^\prime)\|_F = \max_{\|Z\|_F = 1} \langle \nabla f(U) - \nabla f(U^\prime), Z \rangle \\
%& \qquad = \max_{\|Z\|_F = 1} \frac{1}{m}\sum_{i=1}^m \langle A_i, UU^T - X^\star\rangle \langle A_iU, Z\rangle - \langle A_i, U^\prime{U^\prime}^T - X^\star\rangle \langle A_iU^\prime, Z\rangle \\
%& \qquad = \max_{\|Z\|_F = 1} \frac{1}{m}\sum_{i=1}^m \langle A_i, UU^T - U^\prime{U^\prime}^T\rangle \langle A_iU, Z\rangle + \langle A_i, U^\prime{U^\prime}^T - X^\star\rangle \langle A_i(U - U^\prime), Z \rangle\\
%& \qquad \leq (1 + \delta_{2r}) \cdot \max_{\|Z\|_F = 1} \left[\|UU^T - U^\prime {U^\prime}^T\|_F\|ZU^T\|_F + \|U^\prime{U^\prime}^T - X^\star\|_F\|Z(U - U^\prime)^T\|_F\right] \\
%& \qquad \leq (1+\delta_{2r}) \left[ \|(U - U^\prime)U^T + U^\prime (U - U^\prime)^T\|_F\cdot R + (R^2 + \|X^\star\|_F)\|U - U^\prime\|_F \right] \\
%& \qquad \leq (1+\delta_{2r})(3R^2 + \|X^\star\|_F)\|U - U^\prime\|_F,
%%& \qquad \leq (9R^2 + 3\|X^\star\|_F)\|U - U^\prime\|_F
%\end{aligned}
%\end{equation*}
%where the first inequality is due to Fact~\ref{fact:rip}. Similarly, 
By~\eqref{eq:hess-lrm} and Fact~\ref{fact:rip}, we obtain
\begin{equation*}
\begin{aligned}
& \|\nabla^2 f(U) - \nabla^2 f(U^\prime)\| = \max_{\|Z\|_F=1} \left| \mbox{vec}(Z)^T \left( \nabla^2 f(U) - \nabla^2 f(U^\prime) \right) \mbox{vec}(Z) \right| \\
& \qquad = \max_{\|Z\|_F=1} \left| \frac{1}{m} \sum_{i=1}^m 2\langle A_i, UZ^T \rangle^2 - 2\langle A_i, U^\prime Z^T \rangle^2 + \langle A_i, UU^T - U^\prime{U^\prime}^T \rangle \langle A_i, ZZ^T \rangle \right| \\
& \qquad \leq (1 + \delta_{2r}) \cdot \max_{\|Z\|_F=1}\left[ \|(U + U^\prime)Z^T\|_F\|(U - U^\prime)Z^T\|_F + \|UU^T - U^\prime{U^\prime}^T\|_F\|ZZ^T\|_F\right] \\
& \qquad \leq (1 + \delta_{2r}) \cdot\left[ (\|U\|_F + \|U^\prime\|_F)\|(U - U^\prime)\|_F + \|(U - U^\prime)U^T + U^\prime (U - U^\prime)^T\|_F\right] \\
& \qquad \leq (1 + \delta_{2r}) \cdot\left[ (\|U\|_F + \|U^\prime\|_F)\|(U - U^\prime)\|_F + \|U - U^\prime\|_F\|U\|_F + \|U^\prime\|_F \|U - U^\prime\|_F\right] \\
& \qquad \leq 4(1+\delta_{2r})R \cdot \|U - U^\prime\|_F \leq 5R\cdot\|U - U^\prime\|_F,
\end{aligned}
\end{equation*}
where we use $\delta_{2r}<\frac{1}{10}$ in the last inequality. The proof is then completed.
\end{proof}

\subsection{Proof of Theorem~\ref{thm:CR-LRM}}
In view of Fact~\ref{fact:CR_LC}, it suffices to prove that under the RIP assumption in Theorem \ref{thm:CR-LRM}, the following statements hold: (i) Assumption~\ref{ass:hess-lip} holds, (ii) Assumption~\ref{ass:eb} holds, (iii) $\mathcal{L}(f(U^k))$ is bounded for some $k\geq 0$, and (iv) $\mathcal{U}$ equals the set of second-order critical points of $f$.

Suppose that $U\in\mathbb{R}^{n\times r}$ is the initial point of Algorithm~\ref{alg:CR}. 
%Define $\bar{R}>0$ as the fourth root of $10r f(U^0) + \frac{3r}{m}\|b\|^2$ and $\mathcal{F}$ as the following closed convex set:
Define
\begin{align*}
\bar{R} & := \left( 10r f(U^0) + \frac{3r}{m}\|b\|^2 \right)^{\frac{1}{4}} >0, \\
\mathcal{F} & := \mathbb{B}(0;2\bar{R}) = \{ U\in\mathbb{R}^{n\times r}: \|U\|_F \leq 2\bar{R}\}.
\end{align*}
Let $U\in\mathbb{R}^{n\times r}$ be an arbitrary point in $\mathcal{L}(f(U^0))$. Let $\lambda_1\geq \cdots \geq \lambda_r \geq 0$ be the eigenvalues of $U^TU$. Then, it holds that $\|UU^T\|_F^2 = \|U^TU\|_F^2= \lambda_1^2 + \cdots + \lambda_r^2$ and $\|U\|_F^2 = \lambda_1 + \cdots + \lambda_r$. Hence, we obtain $\|U\|_F^4 \leq r\|UU^T\|_F^2$. This, together with Definition \ref{defi:rip} and $f(U) \leq f(U^0)$, yields
\begin{equation*}
\begin{aligned}
\|U\|_F^4 \leq r\|UU^T\|_F^2 & \leq r \cdot \frac{1}{(1 - \delta_{2r})m} \|\mathcal{A}(UU^T)\|^2 \\
& \leq \frac{2r}{(1 - \delta_{2r})m}\left( \|\mathcal{A}(UU^T)-b\|^2 + \|b\|^2 \right) \\
& \leq \frac{8r}{1 - \delta_{2r}} f(U^0) + \frac{2r}{(1 - \delta_{2r})m}\|b\|^2 \leq \bar{R}^4,
\end{aligned}
\end{equation*}
where we use $\delta_{2r}<\frac{1}{10}$ in the last inequality. By the definition of $\mathcal{F}$, we have $\mathcal{L}(f(U^0)) \subset \mbox{int}(\mathcal{F})$ and $\mathcal{L}(f(U^0))$ is bounded. In addition, by Proposition~\ref{prop:smooth-lrm}, $\nabla^2f$ is Lipschitz continuous on $\mathcal{F}$ with Lipschitz constant $L = 10\bar{R}$. Hence, statements (i) and (iii) above hold. Note that Proposition~\ref{prop:second-order=global-lrm} implies that statements (ii) and (iv) hold. The proof is then completed.

\section{Numerical Experiments}\label{sec:numerical}
In this section, we apply the CR method to solve nonconvex minimization problems considered in Section \ref{sec:app}. Our primary goal is to validate Theorems~\ref{thm:CR-PR} and~\ref{thm:CR-LRM}, which concern the global convergence of Algorithm~\ref{alg:CR} to the target signals and its local quadratic convergence rate. All experiments are coded in Matlab and run on a Dell desktop with a 3.40-GHz Intel Core i7-3770 processor and 16~GB of RAM. The code to reproduce all the figures and numerical results in this section can be found online: \url{https://github.com/ZiruiZhou/cubicreg_app.git}.

\subsection{Phase Retrieval}
Our setup of the experiments for phase retrieval is as follows. We first generate a complex signal $z^\star\in\CC^n$ from the standard $n$-dimensional complex Gaussian distribution $\CC\mathcal{N}(0,I_n)$, which is considered to be our target signal. Next, we generate the measurement vectors $\{a_j\}_{j=1}^m$ independently and identically from $\CC\mathcal{N}(0,I_n)$ and compute the measurements $\{b_j\}_{j=1}^m$ by assigning $b_j = |a_j^Hz^\star|$ for each $j$. Here $m$ is chosen to be $m = \lceil 3n\log^3(n) \rceil$, which empirically guarantees that the event in Theorem~\ref{thm:CR-PR} holds with overwhelming probability. In addition, the set of target signals is given by $\mathcal{X}^\star = \{ (x^\star \cos \phi - y^\star \sin \phi, \; x^\star \sin \phi + y^\star \cos \phi) : \phi \in [0,2\pi) \}$, where $x^\star$ and $y^\star$ are the real and imaginary parts of $z^\star$, respectively.

We then apply Algorithm~\ref{alg:CR} to solve the resulting optimization problem~\eqref{opt:PR_real}. For the initial point $(x^0,y^0)$, we draw the entries of $x^0$ and $y^0$ independently and identically from the uniform distribution on the interval $[-5,5]$. In the $k$-th iteration of Algorithm~\ref{alg:CR}, we compute the relative error (RE) of the iterate $(x^k,y^k)$, which is defined as
$$\mbox{RE}(k) = \frac{\mbox{dist}((x^k,y^k),\mathcal{X}^\star)}{\|(x^\star,y^\star)\|} .$$
Moreover, we terminate Algorithm~\ref{alg:CR} when $\mbox{RE}(k) < 10^{-8}$. As Theorem~\ref{thm:CR-PR} suggests, with overwhelming probability, $\{\mbox{RE}(k)\}_{k\geq 0}$ converges to $0$ and the local convergence rate is at least quadratic. To validate such result, we present the logarithm of $\{\mbox{RE}(k)\}_{k\geq 0}$ against the number of iterations in Figure~\ref{fig:CR_PR}. Also, the time for reaching a required solution is recorded. It is clear from Figure~\ref{fig:CR_PR} that $\{\mbox{RE}(k)\}_{k\geq 0}$ converges to $0$ and in the final stages of the algorithm, the convergence rate is at least superlinear.

\begin{figure}[t]
\centering
\subfloat[$n = 64$, Time (sec.) = 4.7]{\includegraphics[width = 0.48\linewidth]{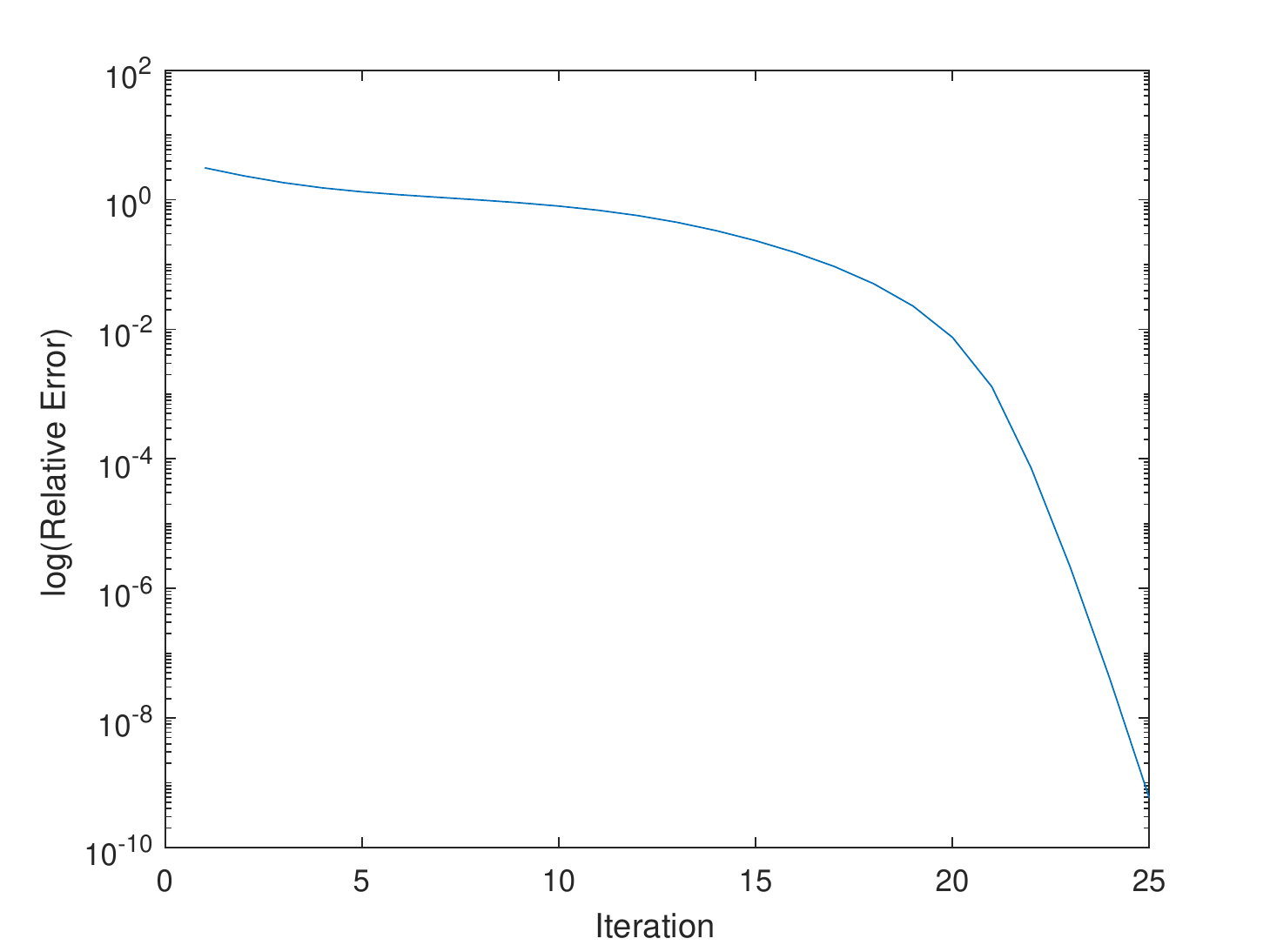}}
\hfill
\subfloat[$n = 128$, Time (sec.) = 12.3]{\includegraphics[width = 0.48\linewidth]{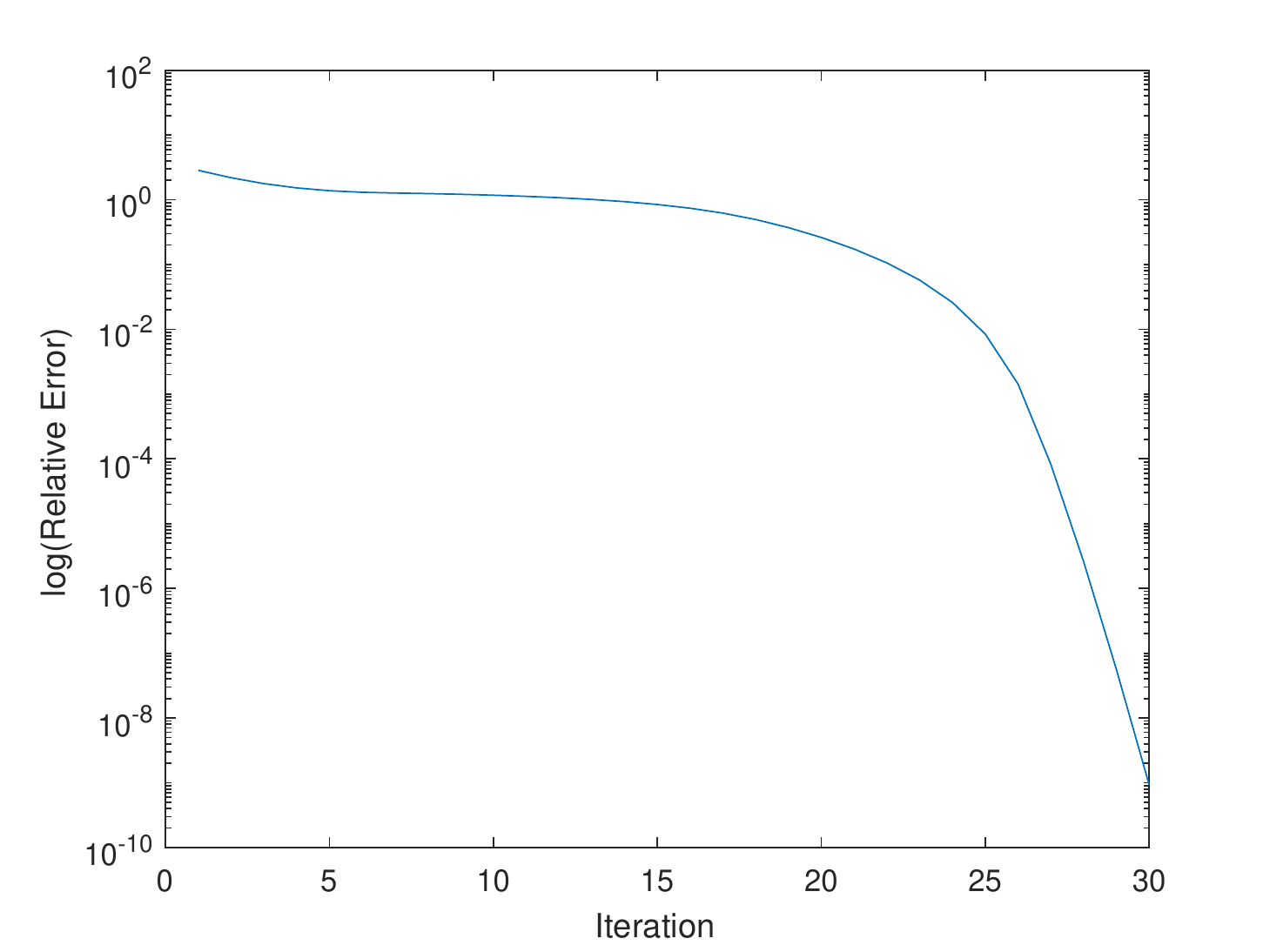}}
\hfill
\subfloat[$n = 256$, Time (sec.) = 129.6]{\includegraphics[width = 0.48\linewidth]{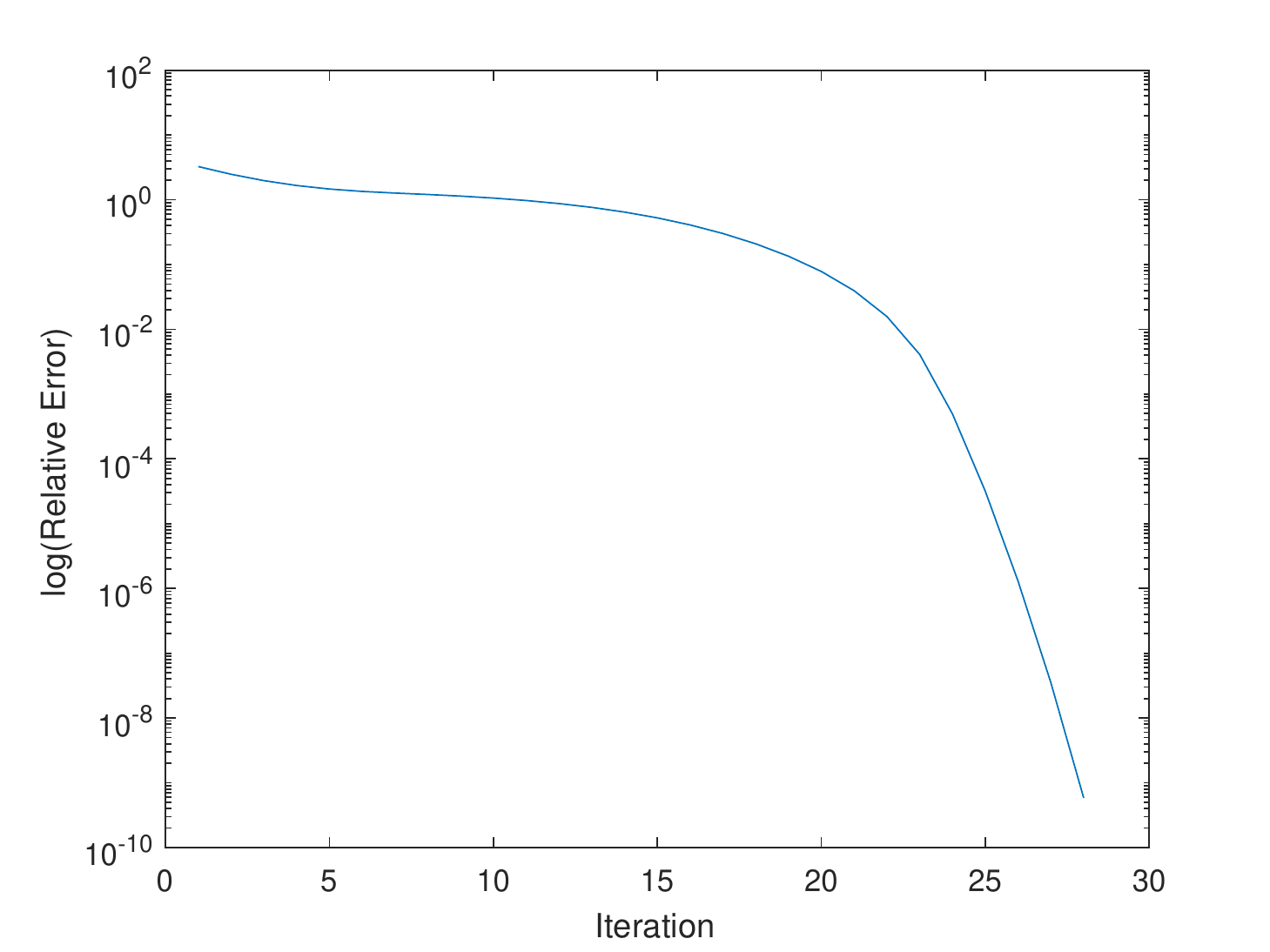}}
\hfill
\subfloat[$n = 512$, Time (sec.) = 577.6]{\includegraphics[width = 0.48\linewidth]{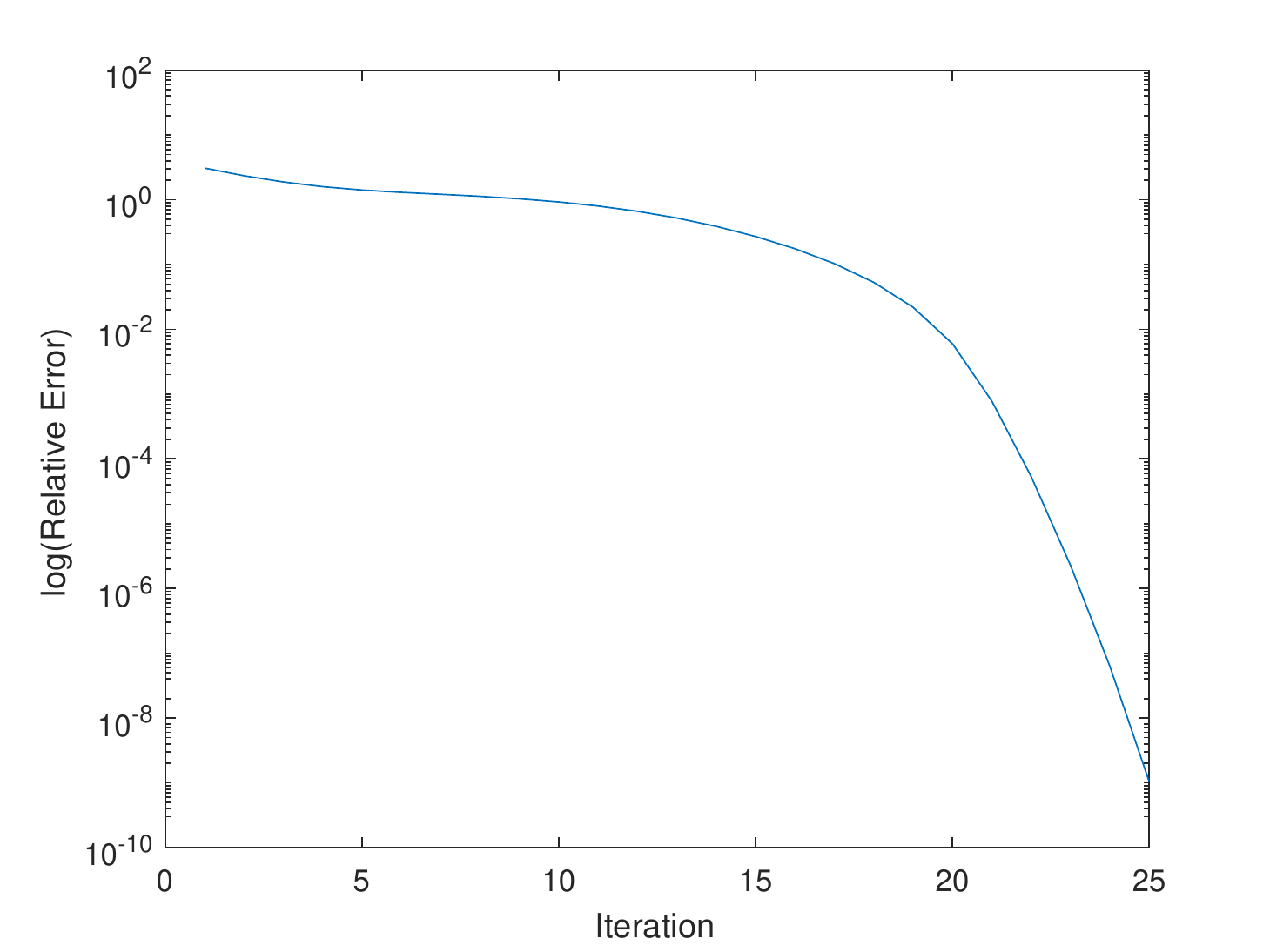}}
\caption{The convergence behaviour of Algorithm~\ref{alg:CR} for solving phase retrieval.}
\label{fig:CR_PR}
\end{figure}

\subsection{Low-Rank Matrix Recovery}
Our setup of the experiments for low-rank matrix recovery is as follows. First, we generate a positive semidefinite matrix $X^\star\in\mathbb{R}^{n\times n}$ with $\mbox{rank}(X^\star) = r$. In particular, we generate a matrix $U^\star\in\mathbb{R}^{n\times r}$ with its entries drawn independently and identically from the standard Gaussian distribution $\mathcal{N}(0,1)$ and set $X^\star = U^\star {U^\star}^T$. Second, we generate the matrices $\{A_j\}_{j=1}^m\subset\mathbb{R}^{n\times n}$ that form the linear operator $\mathcal{A}:\mathbb{R}^{n\times n}\rightarrow\mathbb{R}^m$. For all $i=1,\ldots,m$, entries of $A_i$ are drawn independently and identically from $\mathcal{N}(0,1)$. Here $m$ is chosen to be $m = 3nr$, which empirically guarantees that the event in Theorem~\ref{thm:CR-LRM} holds with overwhelming probability. Finally, we compute the measurements $\{b_j\}_{j=1}^m$ by assigning $b_j = \langle A_j, X^\star\rangle$ for all $j$. In addition, the set of target matrices is given by $\mathcal{U} = \{ U^\star Q: Q\in\mathcal{O}^r\}$.

We then apply Algorithm~\ref{alg:CR} to solve the resulting optimization problem~\eqref{eq:ncvx-lr-mat}. We use a random matrix $U_0\in\mathbb{R}^{n\times r}$, whose entries are drawn independently and identically from the uniform distribution on the interval $[-5,5]$, as the initial point. In the $k$-th iteration of Algorithm~\ref{alg:CR}, we compute the relative error (RE) of the iterate $U^k$, which is defined as 
$$ \mbox{RE}(k) = \frac{\mbox{dist}(U^k,\mathcal{U})}{\|U^\star\|_F}. $$
Moreover, we terminate Algorithm~\ref{alg:CR} when $\mbox{RE}(k) < 10^{-8}$. Note that Theorem~\ref{thm:CR-LRM} implies that, with overwhelming probability, $\{\mbox{RE}(k)\}_{k\geq 0}$ converges to $0$ and the local convergence rate is at least quadratic. To validate such result, we present the logarithm of $\{\mbox{RE}(k)\}_{k\geq 0}$ against the number of iterations in Figure~\ref{fig:CR_LRM}. Also, the time for reaching a required solution is recorded. It is clear from Figure~\ref{fig:CR_LRM} that $\{\mbox{RE}(k)\}_{k\geq 0}$ converges to $0$ and in the final stages of the algorithm, the convergence rate is at least superlinear.

\begin{figure}[t]
	\centering
	\subfloat[$n = 32, r = 6$, Time (sec.) = 9.9]{\includegraphics[width = 0.48\linewidth]{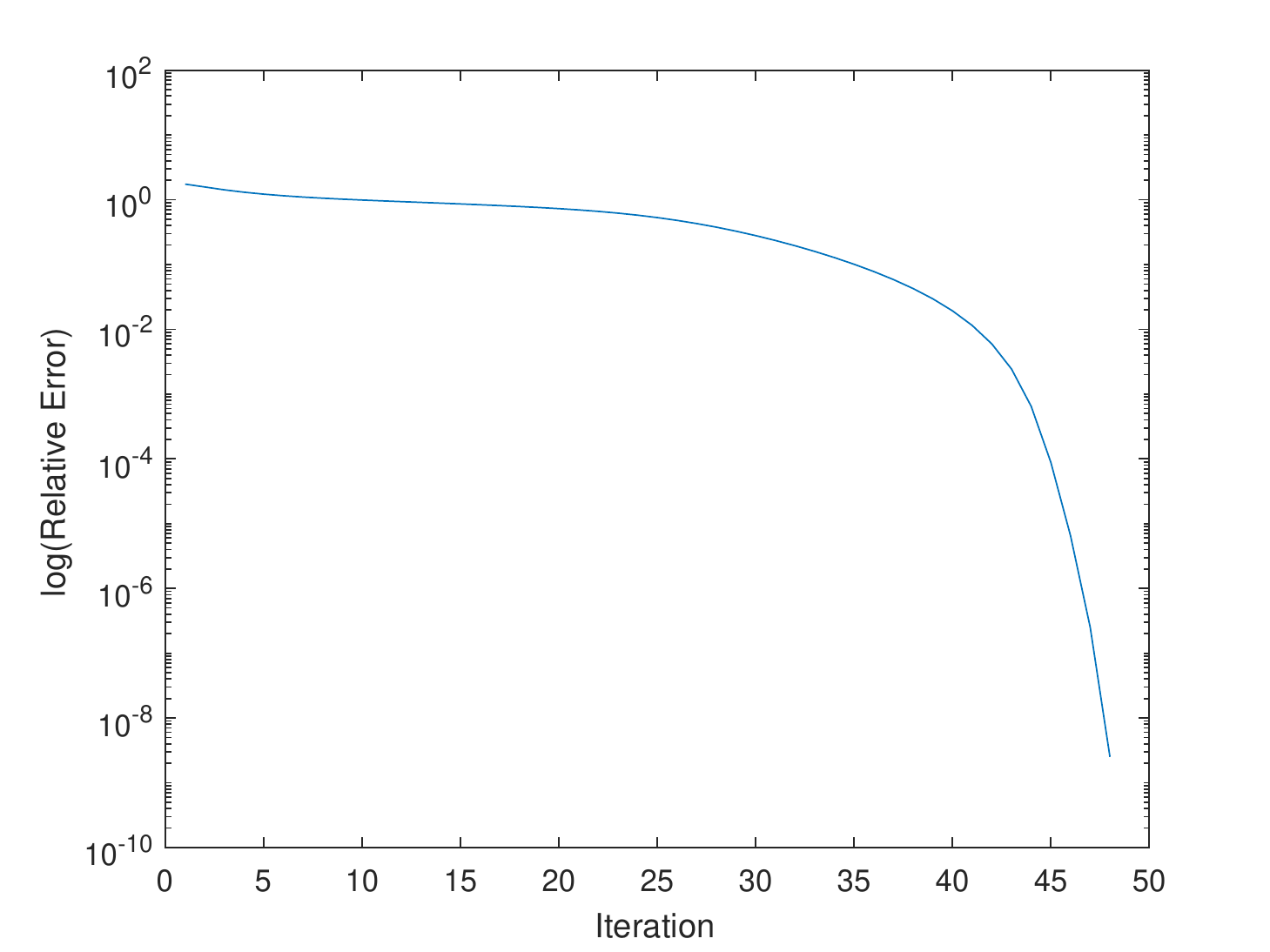}}
	\hfill
	\subfloat[$n = 64, r = 4$, Time (sec.) = 11.9]{\includegraphics[width = 0.48\linewidth]{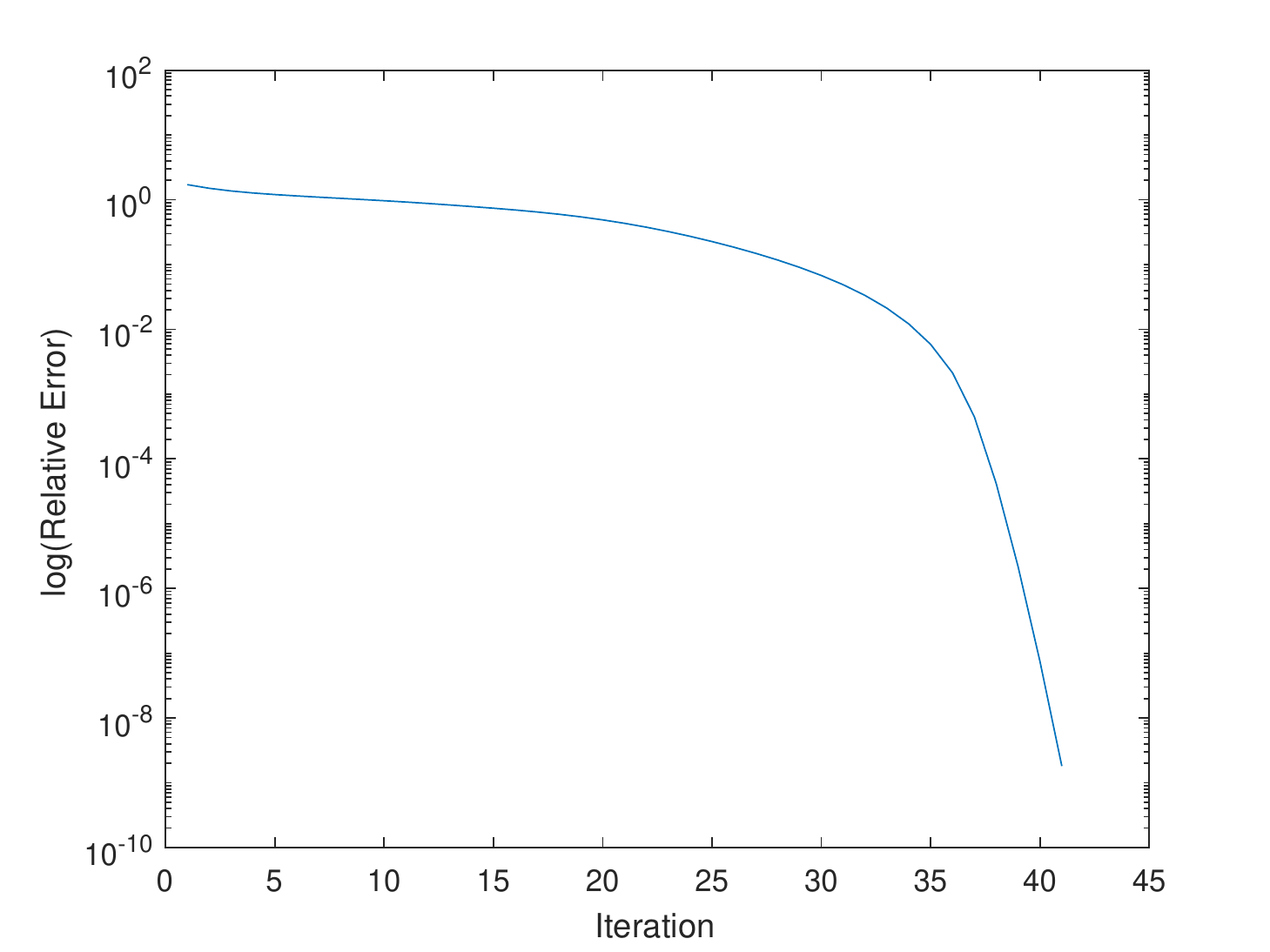}}
	\hfill
	\subfloat[$n = 128, r = 6$, Time (sec.) = 455.0]{\includegraphics[width = 0.48\linewidth]{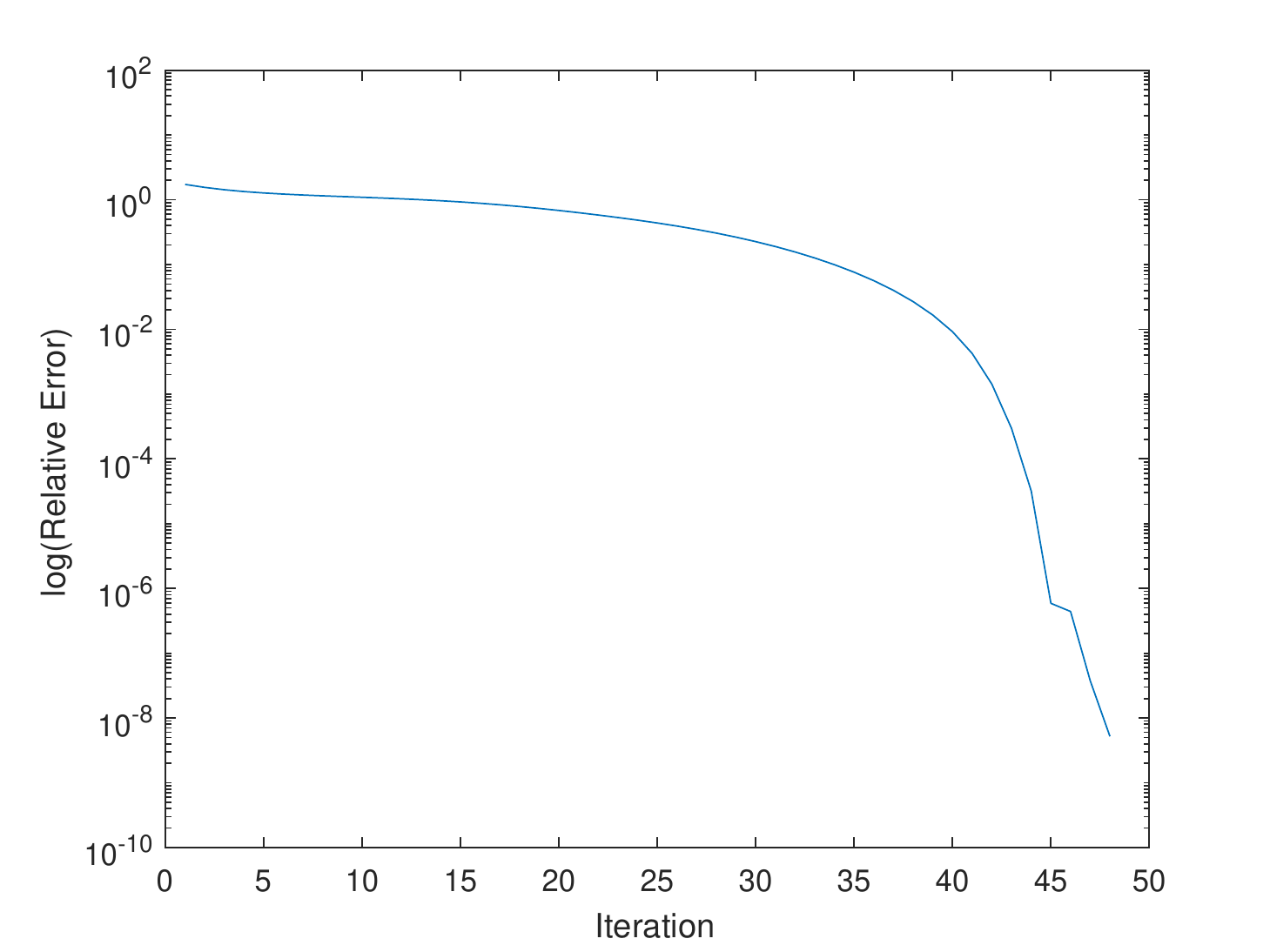}}
	\hfill
	\subfloat[$n = 256, r = 8$, Time (sec.) = 5017.6]{\includegraphics[width = 0.48\linewidth]{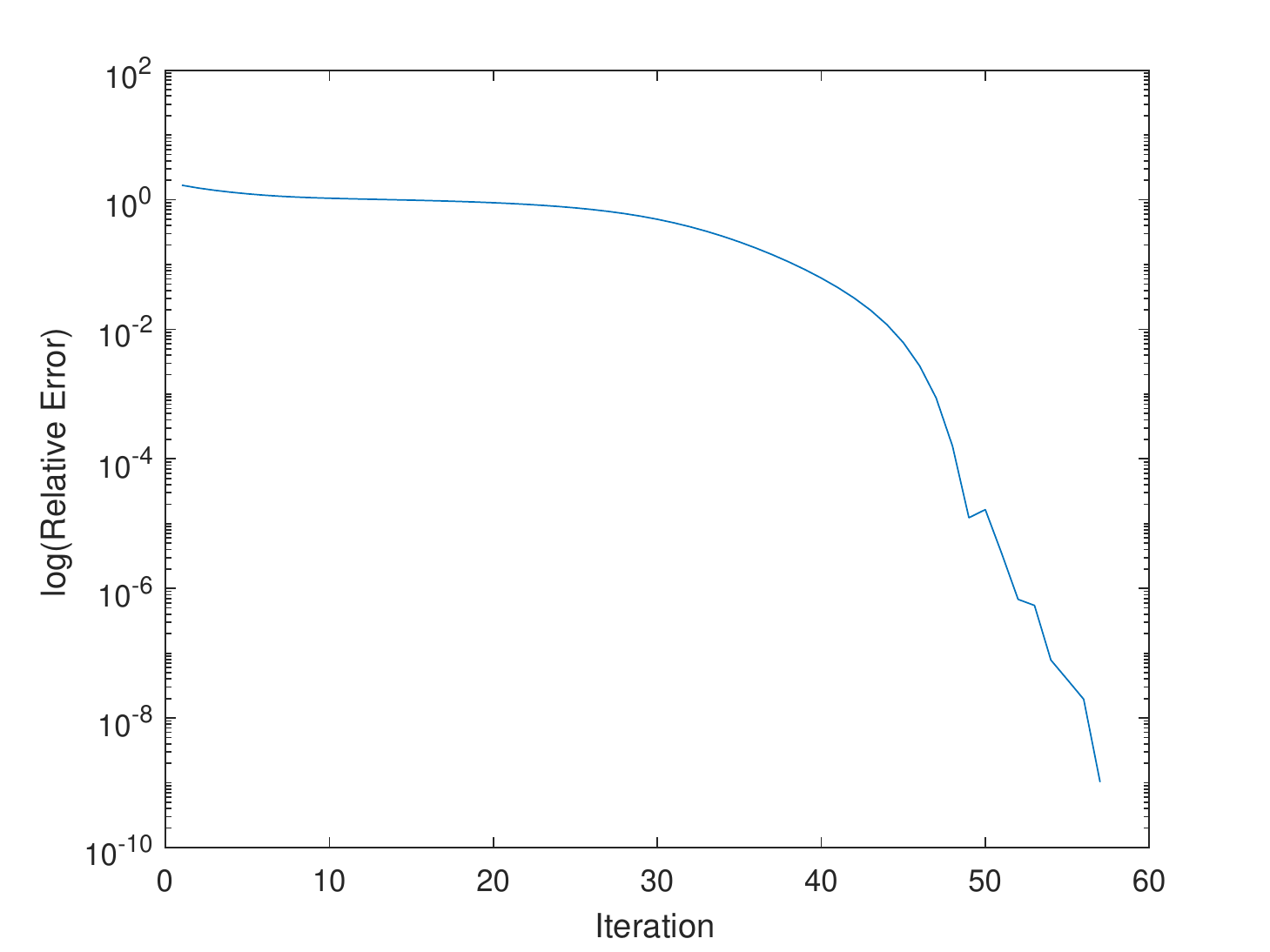}}
	\caption{The convergence behaviour of Algorithm~\ref{alg:CR} for solving low-rank matrix recovery.}
	\label{fig:CR_LRM}
\end{figure}

\section{Conclusions}\label{sec:conclusion}

In this paper we established the quadratic convergence of the CR method under a local EB condition, which is much weaker a requirement than the non-degeneracy condition used in previous works. This indicates that adding a cubic regularization not only equips Newton's method with remarkable global convergence properties but also enables it to converge quadratically even in the presence of degenerate solutions. As a byproduct, we showed that without assuming convexity, the proposed EB condition is equivalent to a quadratic growth condition, which could be of independent interest. In addition, we studied the CR method for solving two concrete nonconvex optimization problems that arise in phase retrieval and low-rank matrix recovery. We proved that with overwhelming probability, the sequence of iterates generated by the CR method for solving these two problems converges at least Q-quadratically to a global minimizer. Numerical results of the CR method for solving these two problems corroborated our theoretical findings.

Our proof of the quadratic convergence of the CR method is not a direct extension of those for other regularized Newton-type methods. The fact that any accumulation point of the sequence generated by the CR method is a second-order critical point plays a key role in our analysis. We believe that similar approaches could be employed for analyzing the local convergence of other iterative algorithms that cluster at second-order critical points, such as trust-region methods.

\bibliographystyle{siamplain}
\bibliography{references}
\end{document}